\documentclass[11pt]{amsart}

\usepackage{amsmath, amssymb, amsthm}
\usepackage{bm}
\usepackage[long,us]{datetime}
\usepackage{enumerate}
\usepackage{fullpage}
\usepackage{graphicx}
\usepackage{tikz}
\usetikzlibrary{positioning,shapes,matrix,arrows}

\usepackage[pdftex,bookmarks,colorlinks,breaklinks]{hyperref}  
\usepackage{color}
\definecolor{dullmagenta}{rgb}{0.4,0,0.4}   
\definecolor{darkblue}{rgb}{0,0,0.4}
\hypersetup{linkcolor=red,citecolor=blue,filecolor=dullmagenta,urlcolor=darkblue} 

\newcommand{\inv}{^{-1}}

\DeclareMathOperator{\supp}{supp}
\newcommand{\weakly}{\rightharpoonup}

\newcommand{\DDD}{D}

\newcommand{\GGG}{G}
\newcommand{\NN}{\mathbb{N}}

\newcommand{\LLL}{\mathcal{L}}

\newcommand{\PPP}{\mathcal{P}}

\newcommand{\RR}{\mathbb{R}}
\newcommand{\ZZ}{\mathbb{Z}}
\newcommand{\veps}{\varepsilon}
\newcommand{\ie}{{\it i.e.}}
\newcommand{\eg}{{\it e.g.}}

\newcommand{\TL}{TL}

\newtheorem{thm}{Theorem}
\newtheorem*{thm*}{Theorem}

\newtheorem{cor}{Corollary}
\theoremstyle{remark}
\newtheorem{lem}{Lemma}
\newtheorem{prop}{Proposition}
\newtheorem{rmk}{Remark}
\newtheorem{defn}{Definition}

\usepackage[backend=bibtex,firstinits=true,style=alphabetic,natbib=true,url=false,sorting=nyt,doi=true,backref=false]{biblatex}
\renewbibmacro{in:}{\ifentrytype{article}{}{\printtext{\bibstring{in}\intitlepunct}}}

\begin{document}

\title{Consistency of Dirichlet Partitions}
\date{\today}

 \author{Braxton Osting}
 \address{Department of Mathematics, 
 University of Utah, 
 Salt Lake City, UT 84112, USA}
 \email{osting@math.utah.edu}
 \thanks{B. Osting and T.~H.~Reeb are partially supported by NSF  DMS-1461138 and NSF DMS 16-19755.}

 \author{Todd Harry Reeb}
 \address{Department of Mathematics, 
 University of Utah, 
 Salt Lake City, UT 84112, USA}
 \email{reeb@math.utah.edu}




\begin{abstract}
  A Dirichlet $k$-partition of a domain $U \subseteq \RR^d$ is a collection of
  $k$ pairwise disjoint open subsets such that the sum of their first
  Laplace-Dirichlet eigenvalues is minimal. A discrete version of Dirichlet
  partitions has been posed on graphs with applications in data analysis. Both
  versions admit variational formulations: solutions are characterized by
  minimizers of the Dirichlet energy of mappings from $U$ into a singular space
  $\Sigma_k \subseteq \RR^k$. In this paper, we extend results of N.\ Garc\'ia
  Trillos and D.\ Slep\v{c}ev to show that there exist solutions of the
  continuum problem arising as limits to solutions of a sequence of discrete
  problems. Specifically, a sequence of points $\{x_i\}_{i \in \NN}$ from $U$
  is sampled i.i.d.\ with respect to a given probability measure $\nu$ on $U$
  and for all $n \in \NN$, a geometric graph $\GGG_n$ is constructed from the
  first $n$ points $x_1, x_2, \ldots, x_n$ and the pairwise distances between
  the points. With probability one with respect to the choice of points
  $\{x_i\}_{i \in \NN}$, we show that as $n \to \infty$ the discrete Dirichlet
  energies for functions $\GGG_n \to \Sigma_k$ $\Gamma$-converge to (a scalar
  multiple of) the continuum Dirichlet energy for functions $U \to \Sigma_k$
  with respect to a metric coming from the theory of optimal transport. This,
  along with a compactness property for the aforementioned energies that we
  prove, implies the convergence of minimizers. When $\nu$ is the uniform
  distribution, our results also imply the statistical consistency statement
  that Dirichlet partitions of geometric graphs converge to partitions of the
  sampled space in the Hausdorff sense.
\end{abstract}

\maketitle

\section{Introduction} \label{sec:intro} The problem of identifying meaningful
groups (``clusters'') within a dataset arises frequently in unsupervised
learning problems, including community detection in sociological networks,
topic modeling, and image segmentation
\cite{jain1999data,xu2005survey,schaeffer2007graph,fortunato2010community,yang2015defining}.
One approach to the clustering problem is to construct a weighted graph,
$\GGG = (V,W)$, where the vertices, $V$, represent the items to be clustered
and a similarity between items is used to define weights, $W$, on the
edges. There is considerable freedom in choosing the weight function and there
exists a large class of methods that can then be used to partition the
resulting graph, \eg, spectral clustering
\cite{Luxburg:2007,ng2002spectral,shi2000normalized,garcia2016spectral},
Dirichlet partitioning \cite{osting2013minimal,ICDM2015,Zosso2015}, and methods
based on minimizing graph cuts (\eg, graph perimeter, Cheeger constant, ratio
cut, balanced cut, normalized cut, {\it etc}\ldots.)
\cite{arora2009expander,Bertozzi2012,vangennip2013mean,bresson2013multiclass,garcia2016cheeger}.

This exploratory approach to clustering can be motivated by the following
statistical model.  Let $U \subseteq \RR^d$ with $d\geq 2$ be a bounded, open
set with Lipschitz boundary and $\nu$ a Borel probability measure with
continuous density $\rho$.  As illustrated in Figure \ref{fig:cartoon}, we
consider a data collection process where we uniformly sample $n$ points,
$x_1, x_2, \ldots, x_n$, from $U$ and construct a geometric graph,
$\GGG_n=(V_n,W^{(n)})$ with $n$ vertices corresponding to $\{x_i\}_{i=1}^n$ and
edge weights $W^{(n)}_{ij}$ for $i,j\in V_n$ that are prescribed functions of
the distances $d(x_i, x_j)$.  We then consider continuum and discrete
partitioning problems on $(U,\nu)$ and $\GGG_n$. In this context, we say that a
partitioning method is \emph{consistent} if the optimal partitions for $\GGG_n$
converge (in the appropriate sense) to the optimal partitions for $(U,\nu)$ in
the large sample limit, $n \to \infty$.  There are several ingredients for a
statistical consistency statement:
\begin{enumerate}
\item[(a)] the continuum and discrete partitioning methods, 
\item[(b)] the construction of the weighted graphs, $\GGG_n$, and 
\item[(c)] the method of comparison between the discrete and continuum partitions. 
\end{enumerate}
An  important consequence for applications is that the partitions obtained using a consistent method will asymptotically stabilize and so the collection of more data will yield diminishing returns. A variety of consistency results have been proven, which we briefly survey in Section \ref{sec:pr}. 

In this paper,  we prove a consistency statement for Dirichlet partitions, which arise in the study of Bose-Einstein condensates 
\cite{bao2004ground,bao2004computing,chang2004segregated} 
and models for interacting agents \cite{conti2002nehari,conti2003optimal,chang2004segregated,cybulski2005,cybulski2008}.  The method of comparison between discrete and continuum partitions used here depends on a metric defined using optimal transport theory, as developed by  Garc\'ia Trillos and Slep\v{c}ev \cite{garcia2014empirical,garcia2016spectral}. This analysis yields practical information about how the graph weights can be constructed and suggests subsampling strategies for extremely large datasets.  

\begin{figure}[t!]
\begin{tikzpicture}[scale=0.9,every node/.style={scale=0.9},thick,>=stealth',dot/.style = {fill = black,circle,inner sep = 0pt,minimum size = 4pt}]
\node (a) at (-0,4) {\includegraphics[width=5cm]{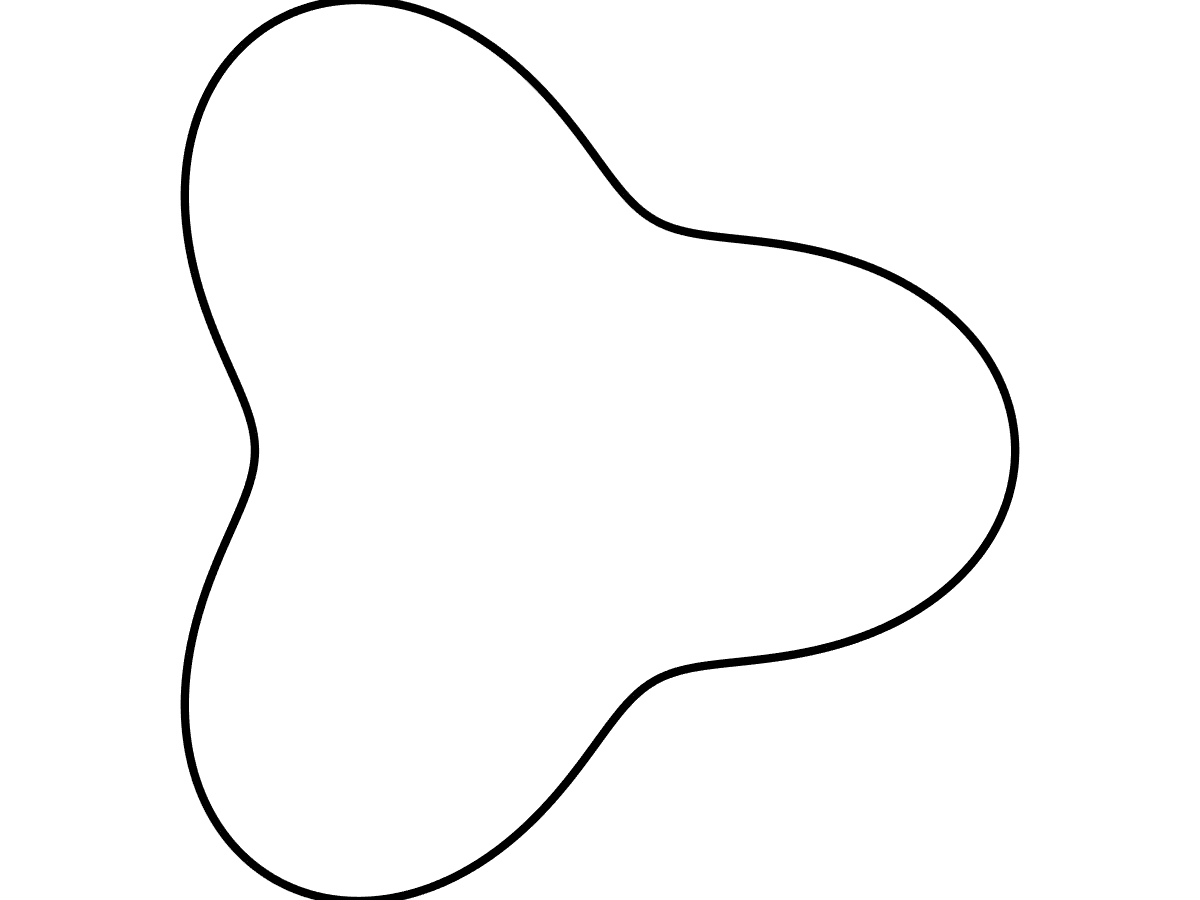}};
\node (b) at (7,4) {\includegraphics[width=5cm]{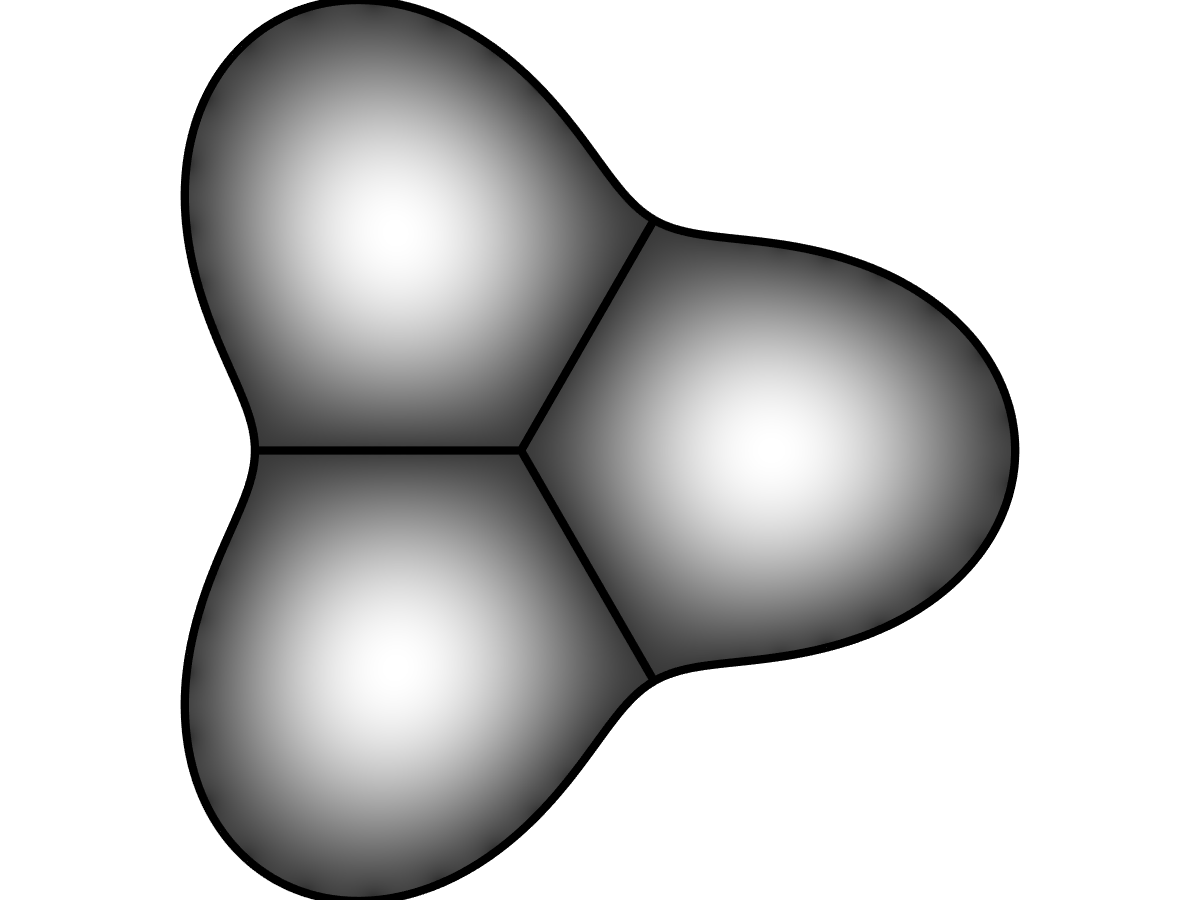}};
\node (c) at (0,-1.7) {\includegraphics[width=6cm]{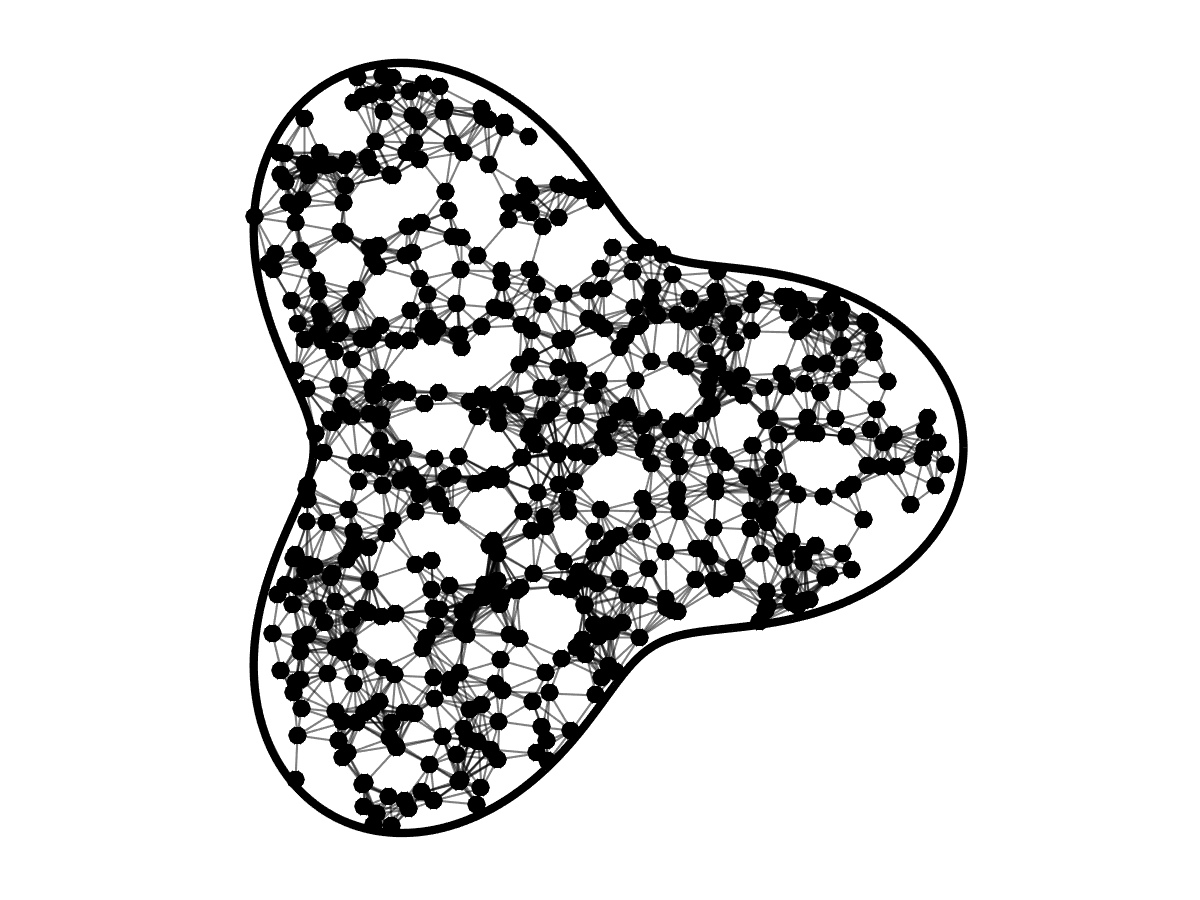}};
\node (d) at (7,-1.7) {\includegraphics[width=6cm]{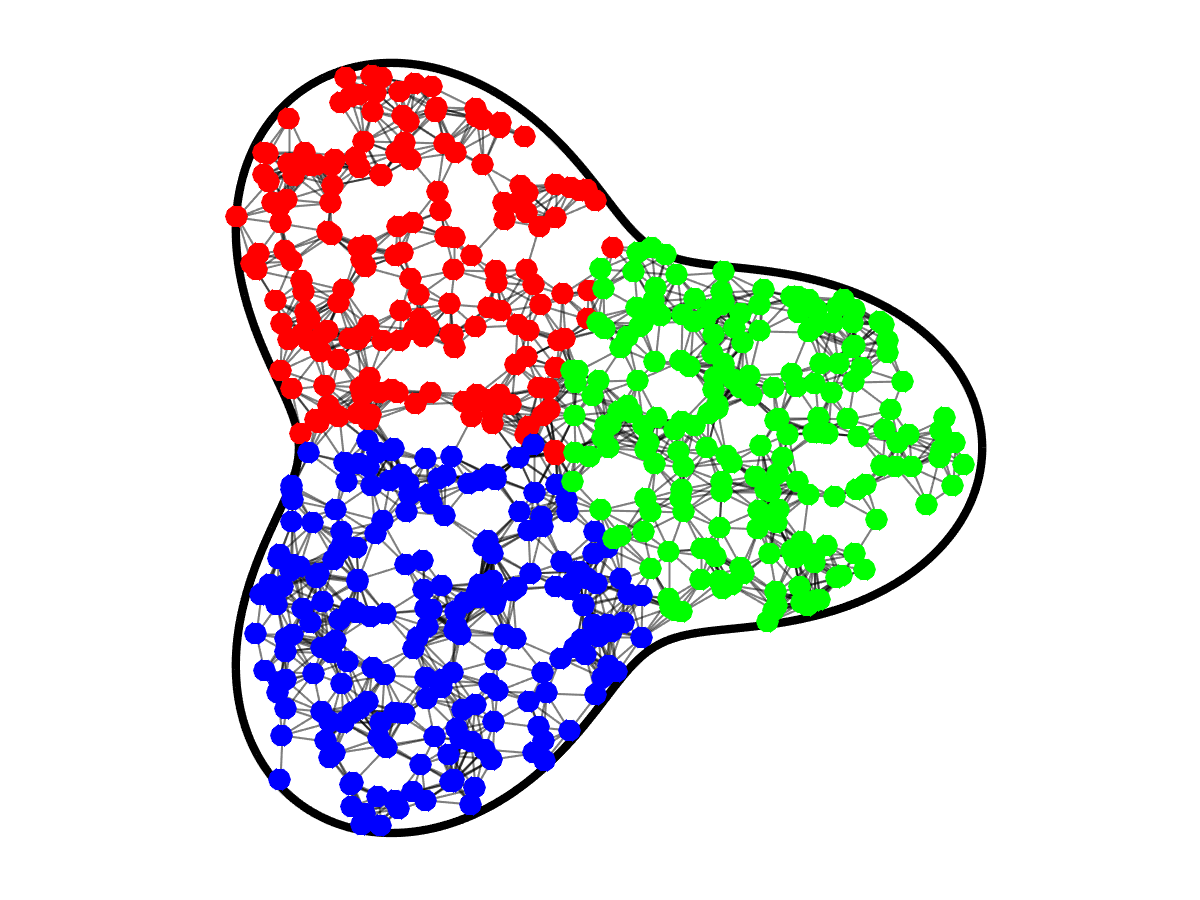}};
\draw (0,3.5) node[label={$(U,\nu)$}](triLat){};
\draw [->]  (2,4) to node [] {\begin{tabular}{c} continuum \\ partition \end{tabular}} (4.8,4);
\draw [->]  (2,-1.7) to node [] {\begin{tabular}{c} discrete \\ partition \end{tabular}} (4.8,-1.7);
\draw [->]  (0,2) to node [left] { \begin{tabular}{c} sample and \\ construct graph \end{tabular}} (0,.5);
\draw [<->]  (7,2) to node [right] { compare} (7,.5); 
\end{tikzpicture}
\caption{Illustration of consistency for the partitioning problem. }
\label{fig:cartoon}
\end{figure}
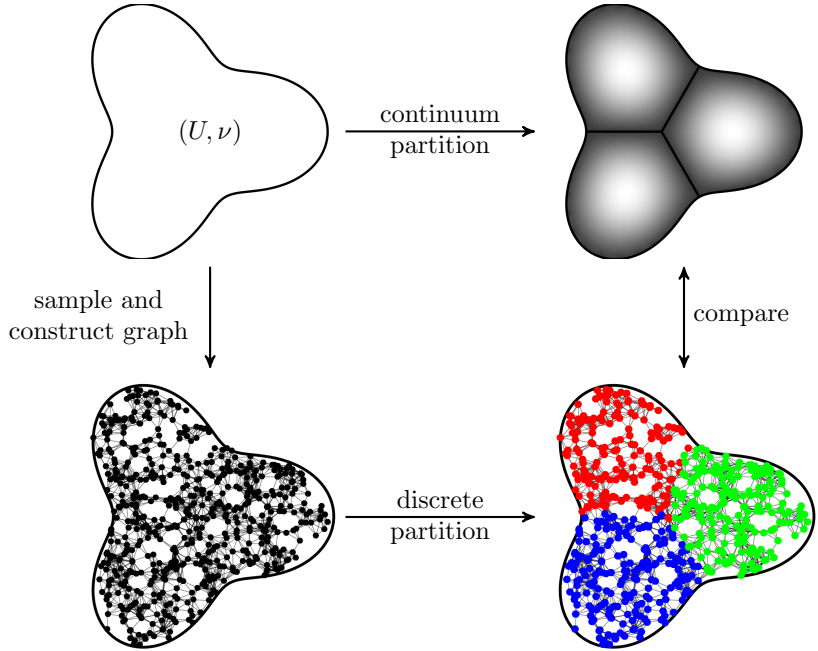  
  
\subsection{Continuum Dirichlet partitions}
Let $U \subseteq \RR^d$ with $d \geq 2$ be an open bounded domain with
Lipschitz boundary. Let $\rho\colon U \to \mathbb R $ be a continuous function
such that there exist constants $M>m>0$ with $m \leq \rho(x) \leq M$ for all
$x \in U$.  The \emph{weighted Dirichlet $k$-partition problem} for
$U \subseteq \RR^d$ is to choose a $k$-partition, \ie, $k$ disjoint quasi-open sets
$U_1, U_2, \ldots, U_k \subseteq U$, that minimize
\begin{equation} \label{eq:ContDirPart}
\sum_{\ell=1}^k \lambda_1(U_\ell)
\end{equation} 
where 
\begin{equation}
  \lambda_1(U) := \min_{\substack{u \in H^1_0(U, \rho)\\ \|u\|_{L^2(U,\rho)}=1}} E(u) 
  \qquad \textrm{and} \qquad 
  E(u) := \begin{cases} 
    \int_U |\nabla u|^2\ \rho^2(x) dx & u \in H^1_0(U, \rho)\\ 
    \infty & \text{ otherwise}
\end{cases}. 
\end{equation}
Here, $E$ is a weighted Dirichlet energy and $\lambda_1(U)$ is the first
Dirichlet eigenvalue of the weighted Laplacian,
$\LLL\colon u \mapsto -\frac{1}{\rho}\text{div}(\rho^2\nabla u)$, on $U$ with
Dirichlet boundary conditions.  The $\rho$-weighted $L^2$-norm is defined
$\|u\|_{L^2(U,\rho)} := \left( \int_U u^2(x)\rho(x)\ dx \right)^{\frac{1}{2}}$.
We refer to any minimizing $k$-partition as a \emph{Dirichlet $k$-partition} of
$(U,\rho)$, or simply a \emph{Dirichlet partition} when $k$ and $(U,\rho)$ are
understood.  Observe that by the monotonicity of Dirichlet eigenvalues, any
Dirichlet partition $\amalg_i U_i$ satisfies
$\overline{U} = \cup_{i=1}^k \overline{U_i}$, which justifies the use of the word ``partition'' in the name. Typically this problem is
considered for $\rho \equiv |\Omega|^{-1}$, in which case $\LLL$ is the Laplacian,
$-\Delta$.  In this setting, the existence of optimal partitions in the class
of quasi-open sets was proved in \cite{bucur1998existence} and, subsequently,
several papers have investigated the regularity of partitions, properties of
optimal partitions, the asymptotic behavior of optimal partitions as
$k \to \infty$, and computational methods
\cite{caffarelli2007optimal,helffer2010remarks,Bonnaillie2007,bourdin2010optimal,Helffer2010b,Helffer2010,Bucur2013,Ramos:2015aa,bogosel2016method}.
        
\bigskip

The Dirichlet partition problem for $U$ is equivalent to the mapping problem
\begin{equation}\label{eq:MapProb}
  \min
  \left\{\mathbf{E}(\mathbf{u}) \colon \mathbf{u} = (u_1, u_2, \ldots, u_k) \in
    H^1_0(U, \rho; \Sigma_k), \int_U u_\ell^2(x) \rho(x) \ dx = 1 \text{ for all } \ell \in [k]\right\}, 
 \end{equation}
where 
\begin{equation} \label{eq:ContDirEnergy}
\mathbf{E}(\mathbf{u}) :=
 \sum_{\ell=1}^k\int_{U}|\nabla u_\ell|^2\rho^2(x)\ dx
 \end{equation} 
 is the (weighted) Dirichlet energy of $\mathbf{u}$,
 $\Sigma_k := \left\{x \in \RR^k \colon \sum_{i \neq j}^k x_i^2x_j^2 = 0
 \right\}$ is the singular space given by the coordinate axes, and
 $ H^1_0(U, \rho; \Sigma_k) = \{\mathbf{u} \in H^1_0(U, \rho; \RR^k) \colon
 \mathbf{u}(x) \in \Sigma_k \text{ a.e.}\} $. We refer to a solution of
 \eqref{eq:MapProb} as a \emph{ground state of $(U,\rho)$}, which, without loss
 of generality, we may assume to be nonnegative.
      
 In particular, if $\mathbf{u}$ is a quasi-continuous representative of a
 ground state such that each component function $u_i$ assumes only nonnegative
 values, then a Dirichlet partition $U = \amalg_i U_i$ is given by
 $U_i = u_i\inv(0,\infty)$ for $i=1,\ldots, k$. Likewise, the first Dirichlet
 eigenvectors $u_i$ of a Dirichlet partition $\amalg_i U_i$ may be assembled
 into a function $\mathbf{u} \in H^1_0(U, \rho;\Sigma_k)$ that solves the mapping
 problem \eqref{eq:MapProb}. These results due to Caffarelli and Lin \cite{caffarelli2007optimal}.
 They used this reformulation to prove regularity results for the case
 $\rho \equiv |\Omega|^{-1}$, such as the locally Lipschitz continuity of
 $\mathbf{u}$ and the $C^{2,\alpha}$-smoothness ($0 < \alpha < 1$) of the
 partition interfaces away from a set of co-dimension two. In particular, for $\rho \equiv |\Omega|^{-1}$, this implies that the Dirichlet partition consists of open sets.   We will use the continuity of $\mathbf{u}$ in this case to establish the Hausdorff convergence of partitions.
 
\subsection{Dirichlet partitions for weighted graphs}
A discrete analogue of the Dirichlet $k$-partition problem has been proposed as
a scheme for clustering data and image segmentation
\cite{osting2013minimal,ICDM2015,Zosso2015}.  We consider a weighted graph $\GGG = (V, W)$ with
vertices $V = \{x_i\}_{i=1}^n$ and symmetric edge weights $W\in \mathbb
R^{n\times n}$, \ie, $W_{ij}$ is the weight of the edge connecting vertices $i$
and $j$ ($i = j$ possibly). The weighted Dirichlet energy of a function $u\colon V \to \RR$ is 
\[E(u) :=\sum_{i,j = 1}^n W_{ij}(u(x_i) - u(x_j))^2.
\] 
The Dirichlet energy of a nonempty
subset $S \subseteq V$ is defined 
\[\lambda_1(S) := \min_{\substack{u|_{S^c} = 0\\   \|u\| = 1}} E(u), \] 
where $u = (u_1, u_2, \ldots, u_n)$ is a function\footnote{We identify the
  space of functions $V\to \mathbb R$ with $\RR^n$, and as in the continuum
  case, the weighted Dirichlet energy is a quadratic functional on the
  appropriate function space.} $V \to \RR$ and $\|u\|^2 = \sum_{i=1}^n u_i^2 $;
as in the continuum case, $\lambda_1(\varnothing) = \infty$.
Defining the degree matrix $D = \mathrm{diag}(d)$, where $d_i = \sum_{j = 1}^n
W_{i,j}$, we have that $\lambda_1$ is the first eigenvalue of the $S\times S$
principal submatrix of the un-normalized graph Laplacian, $D-W$.  For fixed
$k$, the discrete Dirichlet $k$-partition problem on $\GGG$ is then to choose
disjoint subsets $V_1, V_2, \ldots, V_k \subseteq V$ minimizing
$\sum_{\ell=1}^k \lambda_1(V_\ell)$. As with the continuous problem, the
monotonicity of the Dirichlet energy of a subset ensures that a Dirichlet
$k$-partition of $\GGG$ is indeed a partition of $V$ into $k$ disjoint
subsets. In this paper, we will consider Dircichlet partitions for certain
geometric graphs.

\subsection{Construction of geometric graphs}
Let $U \subseteq \RR^d$ with $d \geq 2$ be an open bounded domain with
Lipschitz boundary.  Let $\Omega \supsetneq U$ be a domain in $\RR^d$
satisfying the same properties and that also compactly contains $U$.  We refer
to $\Omega$ as an \emph{auxiliary domain} and require it so that Dirichlet
boundary conditions can be enforced in the discrete problem.  We assume that
$\Omega$ is endowed with the Borel probability measure $\nu$ with density
$\rho$ satisfying the aforementioned conditions.

\begin{rmk} Our setting differs slightly from that of \cite{garcia2016spectral}
  in that the domain of interest is $U$, but we need an auxiliary domain
  $\Omega \supsetneq U$ to enforce Dirichlet boundary conditions for the
  discrete problem. In Section~\ref{sec:Zaremba}, we will consider a modified
  partition model that doesn't require an auxillary domain.
\end{rmk}

As in \cite{garcia2016spectral}, we form a sequence of weighted geometric
graphs $\GGG_n = (V_n, W^{(n)})$ from the first $n$ points
$V_n = \{x_1, x_2, \ldots, x_n\}$ of a sequence of random points
$\{x_n\}_{n \in \NN}$ of $\Omega$ sampled uniformly and independently.  The edge
 incident to vertices  $x_i$ and $x_j$ ($i = j$ possibly) has weight
\[W_{ij}^{(n)} = \eta_{\veps_n}(x_i - x_j) =
  \frac{1}{\veps_n^d}\eta\left(\frac{x_i - x_j}{\veps_n}\right)\] where
$\eta\colon \RR^d \to[0,\infty)$ and $\epsilon_n >0$. We will assume that
$\eta$ is a similarity kernel and $\{\veps_n\}_{n \in \NN}$ is an admissible
sequence, defined as follows.
    
\begin{defn}[\cite{garcia2016spectral}] \label{def:SimKer} We say $\eta\colon
  \RR^d \to [0,\infty)$ is a \emph{similarity kernel}, if there exists a
  profile $\bm{\eta}\colon [0, \infty) \to [0, \infty)$, \ie, $\eta(x) =
  \bm{\eta}(\|x\|)$ for all $x \in \RR^d$, that satisfies the properties
\begin{enumerate}
\item[($\eta$1)] $\bm{\eta}(0) > 0$ and $\bm{\eta}$ is continuous at $0$, 
\item[($\eta$2)] $\bm{\eta}$ is non-increasing, and 
\item[($\eta$3)] $\eta$ has finite  \emph{surface tension}, $\sigma_\eta$, defined 
\begin{equation} \label{eq:ST}
\sigma_\eta :=  \int_{\RR^d}\eta(h)|h_1|^2\ dh. 
\end{equation}
\end{enumerate}
\end{defn}

\begin{defn}[\cite{garcia2016spectral}] \label{def:AdSeq} We say a sequence of
  positive numbers $\{\veps_n\}_{n \in \NN}$ is \emph{admissible} if
  $\epsilon_n \to 0$ and
\[ 
\lim_{n \to \infty} \frac{(\log n)^{C_d}}{n^{1/d}}\frac{1}{\veps_n} = 0 
\qquad \mathrm{where} \quad 
C_d = 
\begin{cases} 
3/4 & d = 2 \\ 
1/d & \mathrm{otherwise} 
\end{cases}.
\]
\end{defn}

\subsection{Dirichlet partitions of geometric graphs} 
For a geometric graph, $G_n = (V_n,W^{(n)})$, as constructed  above, we define the class of $L^2$ vertex functions which vanishes on $\Omega \setminus U$, 
\[
L^2_U(V_n) := \{ u\colon V_n \to \RR \colon u(x_i) = 0 \text{ if } x_i \in \Omega \setminus U\}.
\] 
The Dirichlet energy of a subset $S \subseteq V_n$ is defined to be 
\begin{equation} \label{eq:DiscEnergyGeoGraph}
  \lambda_1(S) := \min_{\substack{ u \in L^2_U(V_n) \\ u|_{S^c} = 0\\
      \|u\|_{\nu_n} = 1}} E(u), 
    \end{equation}
where $u = (u_1, u_2, \ldots, u_n)$ is a function $V \to \RR$ and $\| u \|_{\nu_{n}} = \sqrt{ \frac{1}{n} \sum_{i=1}^n u(x_i)^2 }$ is a weighted variant of the $L^2$-norm (associated with the empirical distribution $\nu_n$ on the first $n$ points). We consider a weighted
variant of the discrete Dirichlet partition problem: choose a partition $V_1
\sqcup V_2 \sqcup \ldots \sqcup V_k$ by disjoint sets to minimize 
\begin{equation} \label{eq:PartGeoGraph}
\sum_{\ell=1}^k\lambda_1(V_\ell), 
\end{equation}
where $\lambda_1(S)$ is defined as in
\eqref{eq:DiscEnergyGeoGraph}.

The discrete Dirichlet partition problem \eqref{eq:PartGeoGraph} has a
formulation analogous to the mapping formulation given in \eqref{eq:MapProb}.
We define
\begin{equation} \label{e:L2U}
L^2_U(V; \Sigma_k) := \{\mathbf{u}\colon V \to \RR^k \mid
\mathbf{u}(V) \subseteq \Sigma_k \text{ and } \mathbf{u}(x_i) = 0 \text{ if }
x_i \notin U\}
\end{equation}
and the \emph{weighted discrete Dirichlet energy} of $\mathbf{u} = (u_1, u_2, \ldots,
u_k) \in L^2_U(V; \Sigma_k)$ by 
\begin{equation} \label{eq:DiscDirEn}
\mathbf{E}_{n, \veps}(\mathbf{u}) := \frac{1}{n^2\veps^2} \sum_{\ell=1}^k E(u_\ell)
\end{equation}
Here, the multiplicative factor is included for the $\Gamma$-convergence of the discrete Dirichlet energy to the continuous energy. 
The discrete mapping problem formulation is then
\begin{equation} \label{eq:DiscMapProb} \min
  \left\{\mathbf{E}_{n,\veps}(\mathbf{u}) \colon \mathbf{u} = (u_1, u_2,
    \ldots, u_k) \in L^2_U(V; \Sigma_k),\ \|u_\ell\|_{\nu_n} = 1 \text{ for all
    } \ell \in [k] \right\},
\end{equation} 
which is clearly equivalent to the discrete Dirichlet partition problem
\eqref{eq:PartGeoGraph}. Again, we refer to minimizers of
\eqref{eq:DiscMapProb} as \emph{ground states} of the graph $G_n$ and, without
loss of generality, assume that $\mathbf{u}$ has nonnegative components.

\subsection{Statement of results} Our main result is the following Theorem,
which states that the discrete Dirichlet energy \eqref{eq:DiscDirEn} for
geometric graphs $\Gamma$-converges to (a constant multiplicative factor of)
the continuum Dirichlet energy \eqref{eq:ContDirEnergy}. The metric used in the
$\Gamma$-convergence is the $TL^2$ metric, which will be defined in Section
\ref{sec:OptTransport}.

\begin{thm}[$\Gamma$-convergence of Dirichlet energies] \label{thm:GconvDE} Let
  $\{\veps_n\}_{n \in \NN}$ be an admissible sequence. Suppose that
  $\{x_n\}_{n \in \NN}$ are sampled i.i.d.\ from $(\Omega,\nu,\rho)$, with
  which we create a sequence of geometric graphs using a similarity kernel,
  $\eta$ as described above. Then with probability one, as $n \to \infty$, the
  sequence of weighted discrete Dirichlet energies
  $\{\mathbf{E}_{n, \veps_n}\}_{n \in \NN}$ $\Gamma$-converges to
  $\sigma_\eta\mathbf{E}$ in the $\TL^2$-sense. Moreover, the compactness
  property also holds for $\{\mathbf{E}_{n, \veps_n}\}_{n \in \NN}$ with
  respect to the $\TL^2$-metric, \ie, every sequence
  $\{\mathbf{u}_n\}_{n \in \NN}$ such that
  $\mathbf{u}_n \in L^2_U(V_n;\Sigma_k)$ with
    \[\sup_{n \in \NN}\|\mathbf{u}_n\|_{\nu_{n}} < \infty 
    \qquad \text{and} \qquad \sup_{n \in \NN} \mathbf{E}_{n,
      \veps_n}(\mathbf{u}_n) < \infty
  \] is precompact in $\TL^2$.
\end{thm}

Because minimizers of $\sigma_\eta\mathbf{E}$ are
  minimizers of $\mathbf{E}$, and vice versa, Theorem~\ref{thm:GconvDE} implies
  that ground states of the graph $\GGG_n$, $\mathbf{u}_n$, converge in the
  $TL^2$ metric (along a subsequence) to a ground state $\mathbf{u}$ of
  $U$. The supports of the components of $\mathbf{u}_n$ and $\mathbf{u}$ define
  Dirichlet $k$-partitions of the graph and domain $U$, so a natural question
  to ask is in what sense  do these associated partitions converge. The
  following corollary shows that when $\nu$ is the uniform distribution on
  $\Omega$, \ie, $\rho \equiv |\Omega|^{-1}$, the associated Dirichlet $k$-partitions of $G_n$ converge to Dirichlet $k$-partitions of $U$ in the Hausdorff distance.

  \begin{cor}[Hausdorff convergence of Dirichlet
    $k$-partitions] \label{cor:Hconv} With the same assumptions as in Theorem
    \ref{thm:GconvDE} and also that $\rho \equiv |\Omega|^{-1}$, let
    $\{\mathbf{u}_n\}_{n \in \NN}$ be ground states of
    $\{ \GGG_n \}_{n\in \mathbb N}$ so that, after passing to a subsequence,
    $\mathbf{u}_n \overset{\TL^2}{\longrightarrow} \mathbf{u}$ where
    $\mathbf{u} \in H^1_0(U, \rho;\Sigma_k)$ is a continuous ground state of $U$.
    Let $U_{n,\ell} = u_{n,\ell}^{-1}(0,\infty)$ and
    $U_\ell = u_\ell^{-1}(0,\infty)$, so that $\amalg_\ell U_{n,\ell}$ and
    $ \amalg_\ell U_{\ell} $ are Dirichlet $k$-partitions of $G_n$ and $U$,
    respectively. Then $U_{n,\ell}$ converges along the same subsequence to
    $\overline{U_\ell}$ in the Hausdorff distance as $n \to \infty$ for all
    $ \ell \in [k]$.
\end{cor}

Observing that the supports of the components of the discrete and continuum ground states define the (closure) of the partitions, \ie 
$$
U_{n,\ell} = \supp(u_{n,\ell}) \quad \textrm{and} \quad \overline{U_\ell} = \supp(u_{\ell}), 
$$
we read Corollary \ref{cor:Hconv} to state that the supports of the components of the discrete ground states converge to the support of the continuum ground states in the Hausdorff sense.

\begin{figure}[t!]
\begin{center}
\includegraphics[width=5cm]{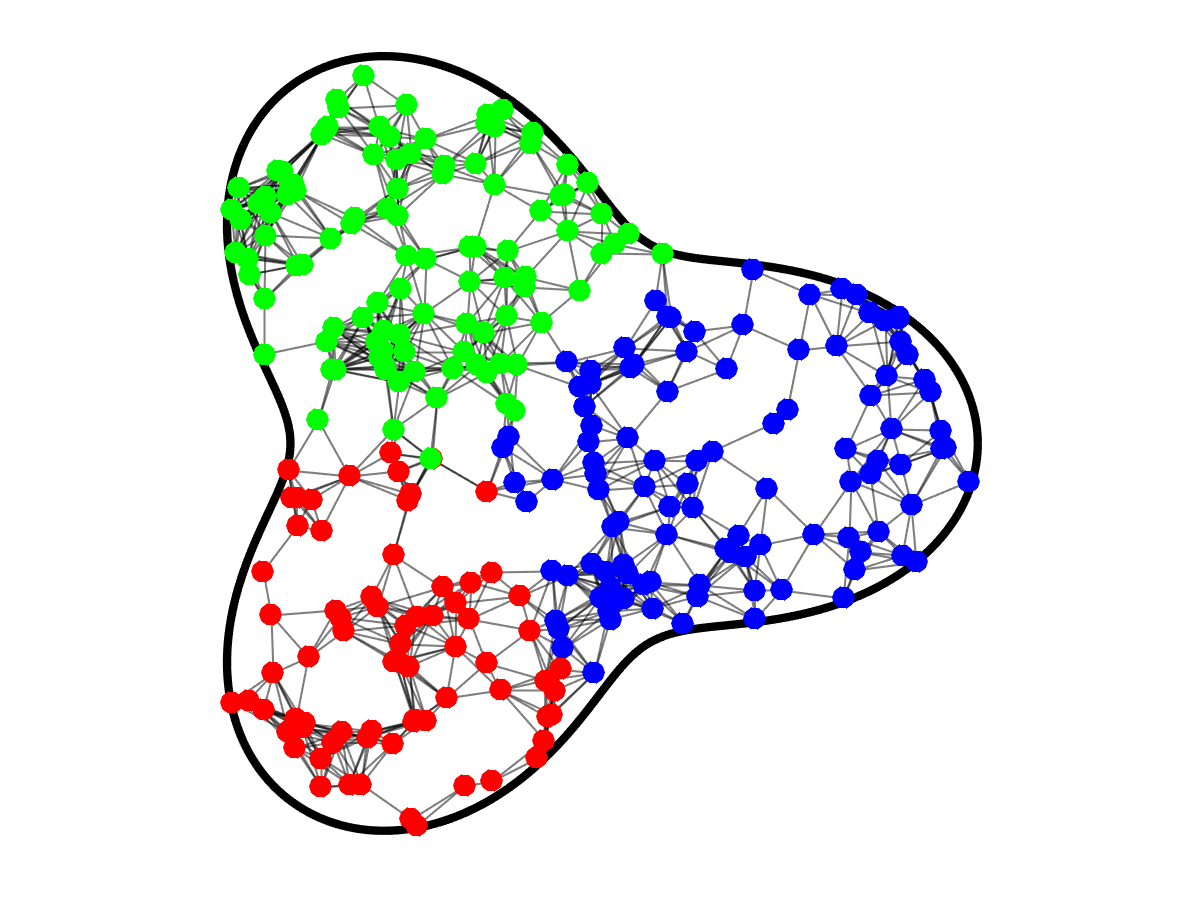}
\includegraphics[width=5cm]{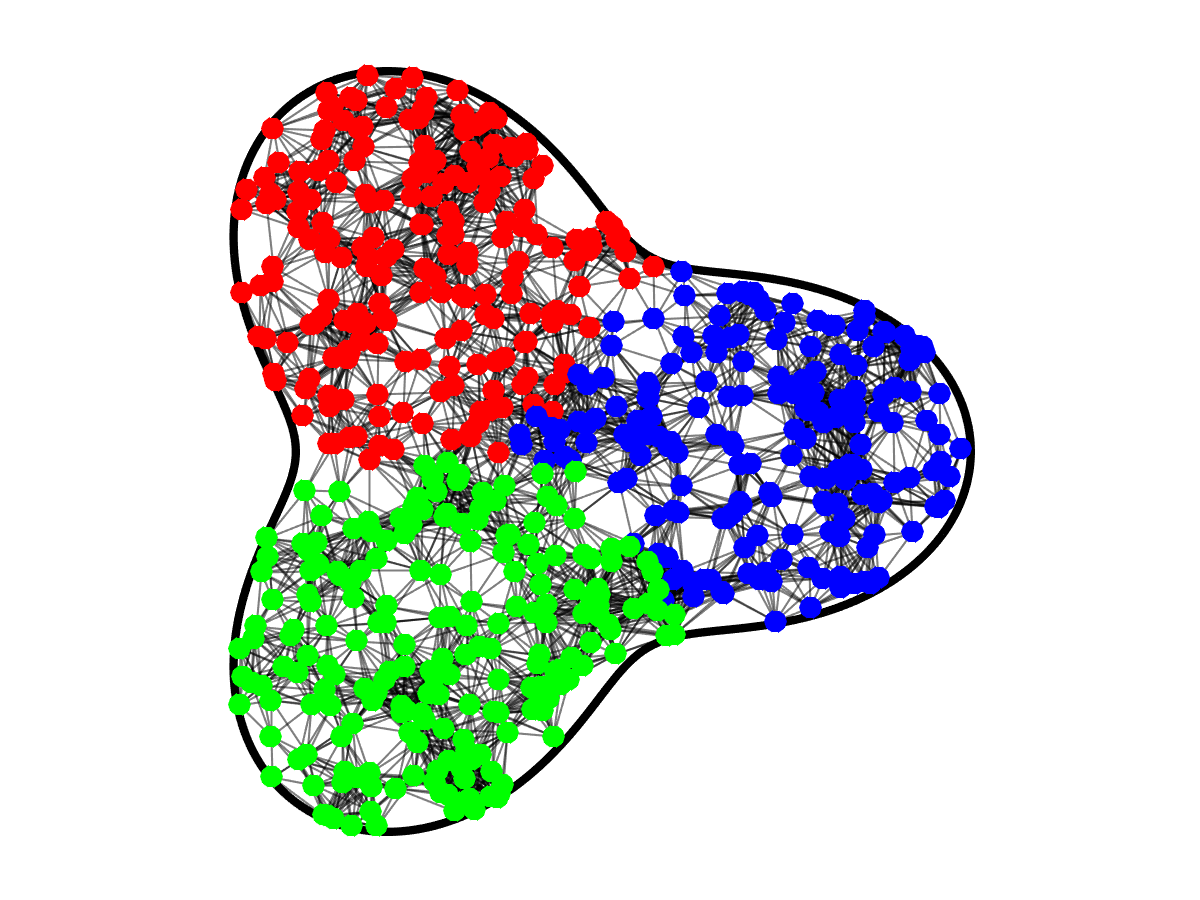}
\includegraphics[width=5cm]{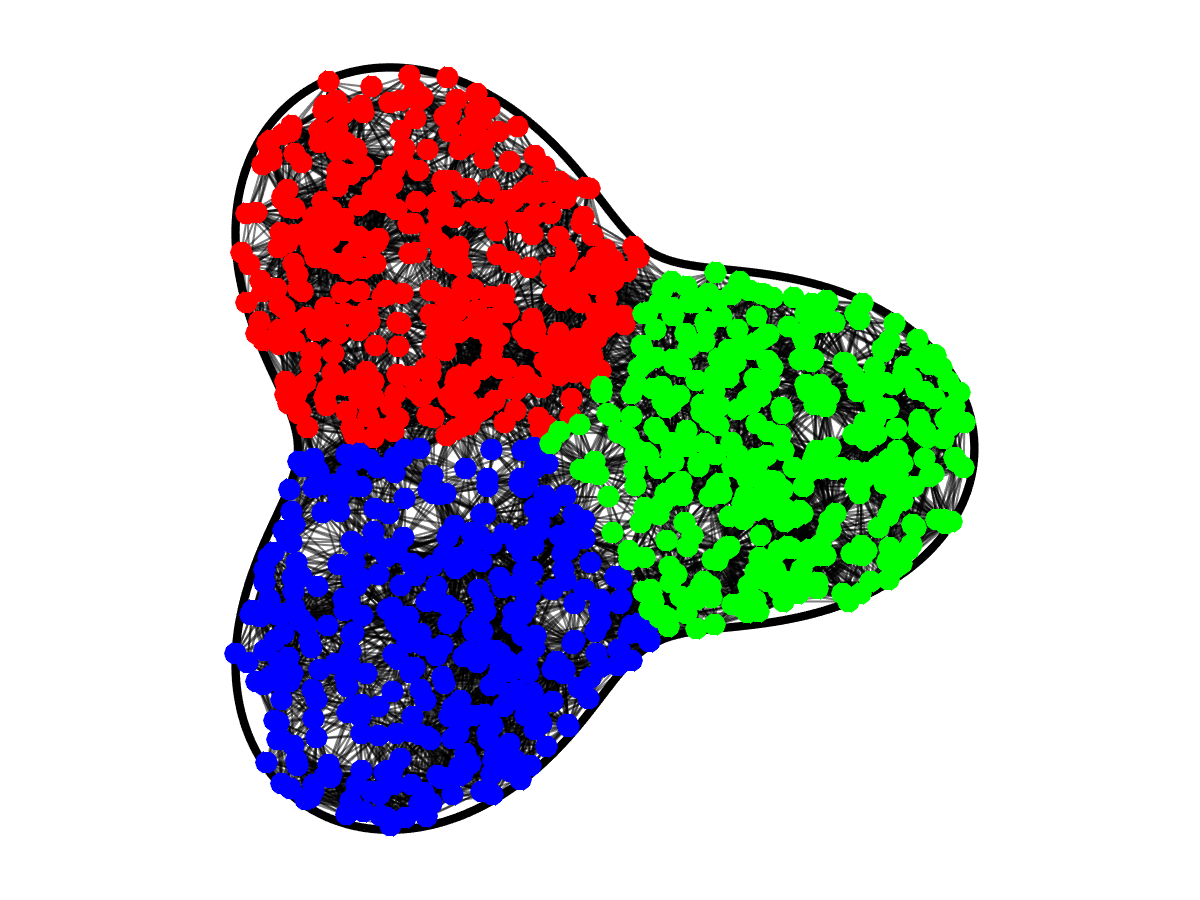}
\end{center}
\caption{An illustration of consistency for the domain in Figure~\ref{fig:cartoon}(top left). Displayed are Dirichlet 3-partitions of a geometric graph constructed from  $n=800$, $1600$, and $3200$ uniformly sampled points. As $n\to \infty$, Corollary~\ref{cor:Hconv} shows that the graph Dirichlet partitions converge to a continuum Dirichlet partition in the Hausdorff sense.  Note that points  sampled from $\Omega \setminus U$ are not displayed and the partition components are colored arbitrarily.  See Section~\ref{sec:CompEx} for computational details. }
\label{fig:conv}
\end{figure}  

Corollary~\ref{cor:Hconv} is illustrated in Figure~\ref{fig:conv}. As the
number of sampled points, $n$, increases, the partition converges (in the
Hausdorff sense) to the partition in the top-right panel of Figure
\ref{fig:cartoon}. The details of the methods used to generate Figures
\ref{fig:cartoon} and \ref{fig:conv} are briefly described in
Section~\ref{sec:CompEx}.

\begin{rmk}
  As observed  in \cite[Section 5.2]{garcia2014tv}, the only place that the
  choice of points $\{x_n\}_{n \in \NN}$ enters in the proofs of
  $\Gamma$-convergence (Theorem~\ref{thm:GconvDE}) is in controlling the rate of weak convergence of the
  empirical measures $\{\nu_n\}_{n \in \NN}$ to $\nu$. This control may be verified directly
  for other sequences, such as the grid points $\Omega \cap
  \bigcup_{r=1}^\infty\frac{1}{2^r}\ZZ^d$, and therefore Theorem
  \ref{thm:GconvDE} holds for these sequences as well.
\end{rmk}

\begin{rmk}\label{rmk:normalized}
  We briefly remark that an analogue of Theorem~\ref{thm:GconvDE}, with
  essentially the same proof, holds when the unweighted graph Laplacian $D - W$
  is replaced by the symmetric normalized graph Laplacian
  $I - D^{-1/2}WD^{-1/2}$. As per \cite{garcia2016spectral}, the discrete and
  continuum weighted energies become
  $$u \mapsto \sum_{i,j=1}^n W_{i,j}\left(\frac{u(x_i)}{\sqrt{D_{i,i}}} -
    \frac{u(x_j)}{\sqrt{D_{j,j}}}\right)^2 \quad \text{ and } \quad u \mapsto
  \int_U\left|\nabla\left(\frac{u}{\sqrt{\rho}}\right)\right|^2\ \rho^2(x)\ dx,$$
    where the continuum energy corresponds to the operator
  $$\mathcal{N}^{\mathrm{sym}}\colon u \mapsto
  -\frac{1}{\rho^{3/2}}\text{div}\left(\rho^2\nabla\left(\frac{u}{\sqrt{\rho}}\right)\right).$$
\end{rmk}

Another natural question is whether we can obviate the need for the auxiliary
domain $\Omega$ if we replace the Dirichlet boundary conditions on $\partial U$
with Neumann boundary conditions. Doing so would not only simplify the
construction of the geometric graphs, but also justifies further investigation
of these techniques in settings where choosing an auxiliary domain may be
infeasible, \eg, when $U$ is a relatively open subset of an embedded
Riemannian manifold.  We answer this question positively in
Section~\ref{sec:Zaremba}. By taking $\Omega = U$ and building the geometric
graphs as before, we have that the discrete energies, which we call
$\mathbf{E}_{n,\veps}^\textrm{Zar}$ $\Gamma$-converge in the $TL^2$-sense to (a
constant multiplicative factor of) a continuum energy $\mathbf{E}^\textrm{Zar}$
corresponding to what we refer to as the \emph{weighted Zaremba partition
  problem}: choose a $k$-partition $U_1,U_2,\ldots,U_k \subseteq U$ that
minimizes
\begin{equation*}
  \sum_{\ell=1}^k \kappa_1(U_\ell).
    \end{equation*}
   Here $\kappa_1(V)$ is the first Zaremba eigenvalue of the weighted
    Laplacian $\mathcal{L}$ with Neumann boundary conditions on
    $\partial V \cap \partial U$ and Dirichlet boundary conditions elsewhere on
    $\partial V$. Specifically, we prove the following analogue of
    Theorem~\ref{thm:GconvDE}.

  \begin{thm}[$\Gamma$-convergence of Dirichlet
    energies] \label{thm:GconvZaremba} Let $\{\veps_n\}_{n \in \NN}$ be an
    admissible sequence. Suppose that $\{x_n\}_{n \in \NN}$ are sampled
    i.i.d.\ from $(U,\nu,\rho)$, with which we create a sequence of
    geometric graphs using a similarity kernel, $\eta$ as described
    before. Then with probability one, as $n \to \infty$, the sequence of
    weighted discrete Dirichlet energies
    $\{\mathbf{E}^{\text{Zar}}_{n, \veps_n}\}_{n \in \NN}$ $\Gamma$-converges to
    $\sigma_\eta\mathbf{E}^{\text{Zar}}$ in the $\TL^2$-sense. Moreover, the
    compactness property also holds for
    $\{\mathbf{E}^{\text{Zar}}_{n, \veps_n}\}_{n \in \NN}$ with respect to the
    $\TL^2$-metric, \ie, every sequence $\{\mathbf{u}_n\}_{n \in \NN}$ such
    that $\mathbf{u}_n \in L^2(V_n;\Sigma_k)$ with
    \[\sup_{n \in \NN}\|\mathbf{u}_n\|_{\nu_{n}} < \infty 
      \qquad \text{and} \qquad \sup_{n \in \NN} \mathbf{E}^{\text{Zar}}_{n,
        \veps_n}(\mathbf{u}_n) < \infty
  \] is precompact in $\TL^2$.
\end{thm}

However, as we show in Section~\ref{sec:Zaremba} with a one dimensional example and in Section~\ref{sec:CompEx} with a two dimensional example, Zaremba partitions have smaller boundary partition components, so the choice between the Dirichlet and Zaremba partitioning models may be application dependent. 

Our paper is organized as follows. 
\subsection{Outline}
In Section~\ref{sec:bg}, we give notation and terminology and describe some
previous results on consistency of clustering.  In
Section~\ref{sec:MainResultsProof}, we describe the mapping problem formulation
for generalized Dirichlet partitions, prove Theorem~\ref{thm:GconvDE}, and
explain how Corollary \ref{cor:Hconv} follows.  In Section~\ref{sec:Zaremba},
we discuss extensions to Zaremba partitions,
where the Dirichlet boundary conditions on $\partial U$ have been replaced with
Neumann boundary conditions. In Section~\ref{sec:CompEx}, we discuss a
numerical method for computing Dirichlet and Zaremba partitions; we also
discuss the qualitative differences between the different partitioning schemes.
We conclude in Section~\ref{sec:discuss} with a brief discussion.

\subsection*{Acknowledgments} \ We would like to thank Dejan Slep\v{c}ev for
helpful discussions and comments, especially with
Lemma~\ref{lem:restrict}. We would also like to thank the anonymous referees for their helpful suggestions. 

\section{Background} \label{sec:bg}

\subsection{Notation and terminology}
We let $\chi_A$ denote the indicator function of the set $A$. We use $\amalg$
to denote the union of disjoint sets. We denote by $[n]$ the set
$\{ i \}_{i=1}^n \subseteq \mathbb N$.

Let $L^p(\Omega,\rho) $ denote the function space with norm, 
$$
\|f\|_{L^p(\Omega,\rho)} := \left( \int_\Omega |f(x)|^p\ \rho(x)\, dx
\right)^{\frac{1}{p}}.
$$
The analogous set of vector-valued functions is denoted 
\begin{align*}
  L^p(\Omega,\rho;\RR^k) :=& \bigoplus_{\ell=1}^k L^p(\Omega,\rho) \\
  =& \{\mathbf{f} = (f_1, f_2, \ldots, f_k)\colon f_\ell \in L^p(\Omega,\rho) \text{ for all } \ell \in [k]\}, 
\end{align*}
with
$\|\mathbf{f}\|_{L^p(\Omega,\rho;\RR^k)} := \left( \|f_1\|_{L^p(\Omega,\rho)}^p +
  \|f_2\|_{L^p(\Omega,\rho)}^p + \cdots + \|f_k\|_{L^p(\Omega,\rho)}^p
\right)^{\frac{1}{p}}$.  Finally we denote the vector-valued Sobolev space by
\begin{align*}
  W^{1,p}(\Omega;\RR^k) := \bigoplus_{i=1}^k W^{1,p}(\Omega)
\end{align*} 
with norm
$\|\mathbf{f}\|_{W^{1,p}(\Omega,\rho;\RR^k)} := \left(
  \|f_1\|_{W^{1,p}(\Omega,\rho)}^p + \|f_2\|_{W^{1,p}(\Omega,\rho)}^p + \cdots
  + \|f_k\|_{W^{1,p}(\Omega,\rho)}^p \right)^{\frac{1}{p}}$, where
$$
\|f\|_{W^{1,p}(\Omega,\rho)} := \left( \int_\Omega |f(x)|^p\ \rho(x)\, dx
 + \int_\Omega |\nabla f(x)|^p\ \rho(x)\, dx\right)^{\frac{1}{p}}.
$$
In particular, we denote $W^{1,2}_0(\Omega,\rho;\RR^k)$ by $H^1_0(\Omega,\rho;\RR^k)$. 
Since $\nu$ and the Lebesgue measure are absolutely continuous with respect to each other, $u \in H^1(\Omega, \rho)$ is equivalent to $u \in H^1(\Omega) $. 

\subsection{Previous results on consistency of clustering} \label{sec:pr}
While  consistency results for some clustering methods in vector spaces, such as $k$-means and single-linkage, have  been  proven  \cite{pollard1981strong,hartigan1981consistency}, less is known  about the consistency of graph-based methods and, in particular, the Dirichlet partitioning method.  
The first approaches to demonstrating consistency  for graph-based methods \cite{belkin2005towards,hein2005graphs,gine2006empirical,singer2006graph,hein2007graph,belkin2008towards,vonluxburg2006consistency} 
 compared discrete and continuum partitions by, using the notation above, restricting continuum functions on $\Omega$ to the graphs $\GGG_n$. 
While intuitive, this approach  requires regularity assumptions on  continuum functions beyond that available  for functions associated with  Dirichlet partitions. 

\subsubsection{Spectral clustering}
Clustering is the problem in unsupervised machine learning concerned with
dividing a set $S$ into a fixed number $k$ of (usually pairwise disjoint)
subsets $S_1, S_2, \ldots, S_k$ such that similar elements of $S$ are in the
same cluster and dissimilar elements are in different clusters. 

A popular class of clustering algorithms is spectral clustering algorithms,
where a data set (read: a finite subset $V$ of a metric space) is viewed as a
graph and embedded into a Euclidean space using its graph Laplacian
eigenvectors, after which a clustering algorithm such as $k$-means is applied
\cite{Luxburg:2007,ng2002spectral,shi2000normalized,garcia2016spectral}. Specifically,
if the weighted graph $\GGG = (V, W)$ has unnormalized graph Laplacian, $\LLL =
\DDD - W$, with eigenvectors $\phi_i$ and we wish to divide $V$ into $k$
clusters, then we first embed $V$ into $\RR^k$ using the map $\Phi_k(x) =
(\phi_1(x), \phi_2(x), \ldots, \phi_k(x))$. The spectral embedding $\Phi_k$ is
believed to preserve the geometry of $V$ well enough so that clustering
$\Phi_k(V)$ gives reasonable clusterings of $V$. The point is that we may apply
to $\Phi_k(V)$ clustering algorithms that require or simply benefit from, say
by easing the analysis of the algorithms, working in a vector space as opposed
to just a metric space. In particular, $k$-means, the standard formulation of
which requires working in a vector space, is usually applied to $\Phi_k(V)$.

\subsubsection{$\Gamma$-convergence of the discrete Dirichlet energies and the consistency of spectral clustering}
Building on previous convergence results relating discrete and continuum
Laplacians \cite{chung1997spectral,grigoryan2009analysis}, Garc\'ia Trillos and
Slep\v{c}ev established the following convergence results for discrete
Laplacians on graphs approximating a domain $\Omega \subseteq \RR^d$.  In
brief, they proved a variational convergence of the Dirichlet energies of a
family of graphs $\{\GGG_n\}_{n \in \NN}$ to a weighted Dirichlet energy on
$\Omega$. They then used this result to prove the consistency of spectral
clustering \cite{garcia2016spectral}.

Specifically, they work in the same setting this paper, albeit with
$\Omega = U$, letting $\Omega \subseteq \RR^d$ $(d \geq 2)$ be an open,
bounded, connected domain with Lipschitz boundary and a Borel probability
measure $\nu$ on $\Omega$ with continuous density $\rho$ such that there exists
constants $m, M > 0$ such that $m \leq \rho(x) \leq M$ for all $x \in
\Omega$. From $(\Omega, \nu, \rho)$, they sample a sequence of random points
$\{x_n\}_{n \in \NN}$ i.i.d.. They form the weighted geometric graphs $\GGG_n$
from the first $n$ points $V_n = \{x_1, x_2, \ldots, x_n\}$ of
$\{x_n\}_{n \in \NN}$ by assigning the edge joining $x_i$ and $x_j$ the weight
\[W_{ij}^n = \eta_{\veps_n}(x_i - x_j) = \frac{1}{\veps^d}\eta\left(\frac{x_i -
      x_j}{\veps_n}\right)\] for a similarity kernel,
$\eta\colon \RR^d \to [0,\infty)$ and admissible sequence $\{\veps_n\}_n$. See
Definitions~\ref{def:SimKer} and \ref{def:AdSeq}.

The weighted Dirichlet energies $E_{n,\veps}$ on $\GGG_n$ and $G$ on $\Omega$
are defined by
\begin{equation}\label{eq:ScalarDirDiscrete}
E_{n,\veps}(u) := \frac{1}{n^2\veps^2}\sum_{i,j=1}^n W_{ij}^n(u(x_i) - u(x_j))^2 \qquad \text{for } u \in L^2(\Omega, \nu_n),
\end{equation}
and
\begin{equation}\label{eq:ScalarDirCont}E(u) := \int_\Omega |\nabla u(x)|^2 \rho(x)^2\ dx \qquad \text{for } u \in
H^1(\Omega, \rho)
\end{equation}
respectively.  Note that $E_{n,\veps}(u_n)  = \frac{2}{n^2\veps^2}\langle
\LLL_{n,\veps}u, u \rangle$, where $\LLL_{n,\veps}$ is the unweighted Laplacian of $G_n$ (given $\eta$ and $\veps_n$). Their convergence result is the following.

\begin{thm}[$\Gamma$-convergence of Dirichlet energies \cite{garcia2016spectral}] \label{thm:ScalarGconvDE}
  Let $\{\veps_n\}_{n \in \NN}$ be an admissible sequence and suppose that
  $\{x_n\}_{n \in \NN}$ are independently sampled from $\Omega$ with respect to
  $\nu$. Then with the probability one, the sequence $\{E_{n, \veps_n}\}_{n \in
    \NN}$ $\Gamma$-converges to $\sigma_\eta E$ as $n \to \infty$ in the
  $\TL^2$-sense. Moreover, the compactness property also holds for $\{E_{n,
    \veps_n}\}_{n \in \NN}$ with respect to the $\TL^2$-metric, \ie, every
  sequence $\{u_n\}_{n \in \NN}$ such that $u_n \in L^2(\Omega,
  \nu_n)$ with  
  \[
  \sup_{n \in \NN} \ \| u_n\|_{\nu_{n}} < \infty 
  \quad \textrm{and} \quad
  \sup_{n \in \NN} \ E_{n, \veps_n}(u_n) < \infty
  \] 
  is precompact in  $\TL^2(\Omega)$.
\end{thm}

Theorem~\ref{thm:ScalarGconvDE} is proven by interpolating with the family of
functionals \[ E_{\veps_n}(u) := \frac{1}{\veps_n^2} \int_\Omega \int_\Omega
\eta_{\veps_n}(x-y)(u(x) - u(y))^2 \rho(x)\rho(y)\ dx\, dy,\] which are the
expectation of the discrete Dirichlet energies $E_{n,\veps_n}$, and then
showing that $\{E_{\veps_n}\}_{n \in \NN}$ $\Gamma$-converge to $\sigma_\eta E$
in the $TL^2$-sense when $\{\veps_n\}_{n \in \NN}$ is an admissible sequence.

\medskip

We remark that a variety of other $\Gamma$-convergence results for graph functionals exist; see \cite{garcia2016cheeger,van2012gamma,Davis2016}.

\subsection{$\Gamma$-convergence} \label{sec:GammaConv}
We now recall the definition of the $\Gamma$-convergence; for alternative
characterizations and other applications of $\Gamma$-convergence, we recommend
\cite{braides2002beginners}.

\begin{defn}[$\Gamma$-convergence] 
  Let $X$ be a metric space. We say that a sequence of functions $\{F_n\colon X
  \to [0,\infty]\}_{n \in \NN}$ \emph{$\Gamma$-converges} to $F\colon X \to
  [0,\infty]$, and refer to $F$ as the $\Gamma$-limit of this sequence, if the
  following properties hold.
  \begin{enumerate}[(1)]
  \item[(a)] $\liminf$ inequality: For all $x \in X$ and all convergent sequences
    $x_n \to x$, \[\liminf_{n \to \infty}F_n(x_n) \geq F(x).\]
  \item[(b)] $\limsup$ inequality: For all $x \in X$, there exists a convergent
    sequence $x_n \to x$ (a \emph{recovery sequence} for $x$) such
    that \[\limsup_{n \to \infty}F_n(x_n) \leq F(x).\]
  \end{enumerate}
\end{defn}

The significance of $\Gamma$-convergence is that together with a compactness
property on the functions $\{F_n\}_{n \in \NN}$, it implies the convergence
(up to a subsequence) of a sequence of minimizers $x_n$ of $F_n$ (if
they exist) to a minimizer $x$ of $F$, similar to how lower semi-continuity and
coerciveness imply the existence of a minimizer in the direct method of
the calculus of variations. 

In practice, the functions $F_n$ may have different domains $X_n$ and/or
$\Gamma$-convergence may require working with an alternative topology. In both
cases, the usual solution is to extend the functions $F_n$ to a common domain
$X$ containing $\bigcup_{n \in \NN}X_n$ by setting $F_n(x) = \infty$ at each of
the new points $x \in X \setminus X_n$ in the same way that the Dirichlet energy of a
function $u \in L^2(\Omega) \setminus  H^1(\Omega)$ is defined to be $\infty$. Observe
that when verifying $\Gamma$-convergence in this case, it suffices to only
consider sequences $\{x_n\}_{n \in \NN}$ where $x_n \in X_n$ for all $n \in
\NN$.

\subsection{The $TL^p$ metric space} \label{sec:OptTransport}

The functionals \eqref{eq:ScalarDirDiscrete} and \eqref{eq:ScalarDirCont}, studied in 
\cite{garcia2016spectral},  are defined on particular spaces of the family
\[\TL^p(\Omega) = \{(\mu, f) \colon \mu \in \PPP_p(\Omega), f \in L^p(\Omega,
\mu)\},\] where $1 \leq p \leq \infty$ and $\PPP_p(\Omega)$ is the set of Borel
probability measures  on $\Omega$ with finite $p$-th moments, $\int_\Omega
|x-y|^p\ d\mu(x)$, for all $y \in \Omega$. Note that $\PPP_p(\Omega)$ is the
same as the set $\PPP(\Omega)$ of all Borel probability measures on $\Omega$
when $\Omega$ is bounded. The metric on $\TL^p(\Omega)$ is defined by
\[d_{\TL^p}((\mu,f), (\nu,g)) = \inf_{\pi \in \Gamma(\mu,
  \nu)}\left(\iint_{\Omega \times \Omega} |x-y|^p + |f(x) - g(y)|^p\
  d\pi(x,y)\right)^{1/p},\] where $\Gamma(\mu,\nu)$ is the set of all couplings
between $\mu$ and $\nu$. This was shown to be a metric in \cite{garcia2014tv}.
In fact, when $\mu$ is absolutely continuous with respect to the Lebesgue
measure, we may rewrite this distance using transportation maps $T \colon
\Omega \to \Omega$ between $\mu$ and $\nu$,
\[
d_{\TL^p}((\mu,f), (\nu, g)) = \inf_{\substack{T \colon \Omega \to \Omega\\
    T_*\mu = \nu}}\left(\int_{\Omega} |x-T(x)|^p + |f(x) - g\circ T(x)|^p\
  d\mu(x)\right)^{1/p};
 \]
see \cite{villani2003topics} for details.

We will work in the vector-valued analogue of $\TL^p(\Omega)$ for $p = 2$,
\[\TL^p(\Omega; \RR^m) := \{(\mu, \mathbf{f}) \colon \mu \in \PPP_p(\Omega),\ \mathbf{f} \in L^p(\Omega, \mu; \RR^m)\},\] which we
will, abusing notation,  call $\TL^p$. Accordingly, the distance in $\TL^p$
is \[d_{\TL^p}((\mu,\mathbf{f}), (\nu, \mathbf{g})) = \inf_{\pi \in \Gamma(\mu,
  \nu)}\left(\iint_{\Omega \times \Omega} |x-y|^p + |\mathbf{f}(x) -
  \mathbf{g}(y)|^p\ d\pi(x,y)\right)^{1/p},\] which is indeed a metric by the
arguments of \cite{garcia2014tv} mutatis mutandis. As above, when $\mu$ is
absolutely continuous with respect to the Lebesgue measure, we may rewrite this
distance using transportation maps $T \colon \Omega \to \Omega$ between $\mu$
and $\nu$: \[d_{\TL^p}((\mu,\mathbf{f}), (\nu, \mathbf{g})) = \inf_{\substack{T
    \colon \Omega \to \Omega\\ T_*\mu = \nu}}\left(\int_{\Omega} |x-T(x)|^p +
  |\mathbf{f}(x) - \mathbf{g}\circ T(x)|^p\ d\mu(x)\right)^{1/p}.
  \] 
  In particular, as proven in \cite[Propositions 3.3 and 3.12]{garcia2014tv},
  convergence $(\mu_n, \mathbf{f}_n) \xrightarrow{\TL^p} (\mu, \mathbf{f})$
  amounts to $\mu_n \weakly \mu$ and
  \[\|\mathbf{f} - \mathbf{f}_n\circ T_n\|^p_{L^p(\Omega,\mu;\RR^k)} =
    \int_\Omega |\mathbf{f}(x) - \mathbf{f}_n\circ T_n (x)|^p\ d\mu(x) \to 0\]
  for any (equivalently every) stagnating sequence
  $\{T_n \colon \Omega \to \Omega\}_{n \in \NN}$ of transportation maps, \ie,
  for any (equivalently every) sequence
  $\{T_n \colon \Omega \to \Omega\}_{n \in \NN}$ of transportation maps such
  that
  \[\int_\Omega |x - T_n(x)|^2\ dx \to 0.\]

Finally, we recall that for the sequence of empirical measures $\{\nu_n\}_{n
  \in \NN}$ constructed from $\{x_n\}_{n \in \NN}$, the stagnating sequence of
transportation maps $\{T_n\}_{n \in \NN}$ where $T_n$ takes $(\Omega, \nu)$ to
$(\Omega, \nu_n)$, \ie, $T_{n*}\nu = \nu_n$ for all $n \in \NN$, may be chosen
so that $T_n$ sends the points of $\Omega$ only to nearby points of
$\{x_i\}_{i=1}^n$. Specifically, Garc\'{i}a Trillos and Slep\v{c}ev have shown
the following.

\begin{thm}[\cite{garcia2014empirical}]\label{thm:linfinity}
  Let the probability space $(\Omega, \nu, \rho)$ be such that
  $\Omega$ is an open, bounded, connected subset of $\RR^d$ ($d \geq
    2$) with Lipschitz boundary and such that $\nu$ is absolutely continuous
    with a density function $\rho \in C^0(\Omega)$ where there are constants $0
    < m < M$ such that $m \leq \rho(x) \leq M$ for all $x \in \Omega$. Let
  $\{x_n\}_{n \in \NN}$ be a sequence of points sampled i.i.d.\ from $(\Omega,
  \nu, \rho)$ and let $\{\nu_n\}_{n \in \NN}$ be the corresponding sequence of
  empirical measures. Then with probability one, there exists a positive
  constant $C > 0$ and a sequence of transportation maps $\{T_n\}_{n \in \NN}$
  from $(\Omega, \nu)$ to $(\Omega, \nu_n)$, \ie, $T_{n*}\nu = \nu_n$ for all
  $n \in \NN$, such that \[\limsup_{n \to \infty} \frac{\|T_n -
    \text{Id}\|_{\infty}}{r(d,n)} \leq C\] where $r(d,n) = \frac{(\log
    n)^{3/4}}{n^{1/2}}$ if $d = 2$ and $r(d,n) = \frac{(\log
    n)^{1/d}}{n^{1/d}}$ if $d \geq 3$.
\end{thm}

This ``$L^\infty$-control'' on the transportation maps  informs the hypotheses on and the proofs of the main results in \cite{garcia2014tv, garcia2016spectral}. Consequently, our results also depend on Theorem~\ref{thm:linfinity} and we will further use it to extend Corollary \ref{cor:Hconv}; see Remark \ref{r:TL2}.

\subsection{The Hausdorff distance}
We now recall the definition of the Hausdorff distance. We recommend
\cite[Section 1.8]{schneider2014convex} for  basic results and applications
of the Hausdorff distance, and \cite[Chapter 2]{henrot2006extremum} for
applications pertaining to eigenvalue problems.

For nonempty subsets $X, Y \subseteq \RR^d$, the Hausdorff distance between $X$
and $Y$ is defined to be 
\[d_H(X,Y) := \inf\{\veps > 0\colon  X \subseteq Y_\veps \textrm{ and } Y \subseteq X_\veps\}\] 
where $X_\veps := \{x \in \RR^d\colon \text{ there exists } y
\in X \text{ such that } d(x,y) < \veps\}$ is an $\epsilon$-neighborhood of $X$. The Hausdorff distance is in fact a
metric on the set of nonempty compact subsets of $\RR^d$. In the context of
eigenvalue problems on subdomains of a domain $X$ with compact closure, the
Hausdorff distance between two nonempty open subsets $U, V \subsetneq X$ is
often defined to be $d_H(U,V) := d_H(X \setminus U, X \setminus V)$. However,
the Hausdorff distance ceases to be a metric when we mix the two cases, \eg,
$d_H(U,\overline{U}) = 0$ even if $U \neq \overline{U}$. To conclude, we
observe that while convergence with respect to the Hausdorff distance is
defined in the obvious way, \ie, $\lim_{n \to \infty}X_n = X$ if $\lim_{n \to
  \infty} d_H(X_n,X) = 0$, the following alternative characterization is useful
for when all $X_n$ are contained in the same compact set $Y$: $X_n
\overset{H}{\to} X$ as $n \to \infty$ if all $x \in X$ is the limit (with
respect to the Euclidean metric) of some sequence $\{x_n\}_{n \in \NN}$ with
$x_n \in X_n$ for all $n \in \NN$ and if $\lim_{n \to \infty}x_n \in X$, again
with respect to the Euclidean metric, for any convergent sequence $\{x_n\}_{n
  \in \NN}$ with $x_n \in X_n$ for all $n \in \NN$.

\section{Proof of Main Results}  \label{sec:MainResultsProof} 

\subsection{Proof of Theorem~\ref{thm:GconvDE}} By the definitions of the
weighted Dirichlet energies \eqref{eq:ContDirEnergy}, for $\mathbf{u} \in L^2_U(V_n;\Sigma_k)$ and
$\mathbf{v} \in H^1_0(\Omega, \rho; \Sigma_k)$, we can write
\begin{subequations}
\label{obs:sum}
\begin{align}
  \mathbf{E}_{n,\veps_n}(\mathbf{u})   &= E_{n,\veps_n}(u_1) + E_{n,\veps_n}(u_2) + \cdots + E_{n,\veps_n}(u_k) \\
  \mathbf{E}(\mathbf{v}) &= E(v_1) + E(v_2) + \cdots + E(v_k).
  \end{align}
  \end{subequations} 
  The functionals $E_{n,\veps_n}$ and $E$ appearing in the right-hand sides in
  \eqref{obs:sum} are the Dirichlet energies for scalar-valued functions, \eqref{eq:ScalarDirDiscrete} and \eqref{eq:ScalarDirCont}. It was proven in \cite{garcia2016spectral} that $E_{n, \veps_n}
  \overset{\Gamma}{\longrightarrow} \sigma_\eta E$ in the $\TL^2$-sense; see
  Theorem~\ref{thm:ScalarGconvDE}.

Since
  $\mathbf{E}_{n,\veps}(\mu, \mathbf{f}) = \infty$ when $\mu \neq \nu_n$ or
  $\mu = \nu_n$ but $\mathbf{f} \notin L^2_U(V_n;\Sigma_k)$ and
  $\mathbf{E}(\mu, \mathbf{f}) = \infty$ when $\mu \neq \nu$ or $\mu = \nu$ but
  $\mathbf{f} \neq H^1_0(\Omega, \rho; \Sigma_k)$, and hence the claims below are
  either trivial or vacuous in these cases, we only consider sequences in
  $\TL^2$ of the form
  $(\nu_n, \mathbf{u}_n) \overset{\TL^2}{\longrightarrow} (\nu, \mathbf{u})$
  with $\mathbf{u}_n \in L^2_U(V_n;\Sigma_k)$ for all $n$ and
  $\mathbf{u} \in H^1_0(\Omega, \rho; \Sigma_k)$.
  
  Theorem~\ref{thm:GconvDE} requires the proof of the liminf inequality, limsup
  inequality, and a compactness result, which we prove in turn.
  
  \subsection*{Liminf inequality} Claim: For all $(\nu_n, \mathbf{u}_n)
  \overset{\TL^2}{\longrightarrow} (\nu, \mathbf{u})$, $\sigma_\eta
  \mathbf{E}(\mathbf{u}) \leq \liminf_{n \to \infty}
  \mathbf{E}_{n,\veps}(\mathbf{u}_n)$.\vskip10pt

  Given $(\nu_n, \mathbf{u}_n) \rightarrow (\nu, \mathbf{u})$, we restrict
  componentwise to get the convergent sequence $(\nu_n, u_{n,\ell}) \rightarrow
  (\nu, u_\ell)$ in $\TL^2(\Omega)$. By the $\Gamma$-convergence of the scalar
  Dirichlet energies (Theorem~\ref{thm:ScalarGconvDE}), we have
  that \[\sigma_\eta E(u_\ell) \leq \liminf_{n \to \infty}
  E_{n,\veps_n}(u_{n,\ell})\] for all $\ell = 1, 2, \ldots, k$ and therefore, using
  \eqref{obs:sum}, $\sigma_\eta \mathbf{E}(\mathbf{u}) \leq \liminf_{n \to
    \infty} \mathbf{E}_{n,\veps_n}(\mathbf{u}_n)$.
  
  \subsection*{Limsup inequality} Claim: For all
  $\mathbf{u} \in H^1_0(\Omega, \rho; \Sigma_k)$, there exists a recovery
  sequence
  $(\nu_n, \mathbf{u}_n) \overset{\TL^2}{\longrightarrow} (\nu, \mathbf{u})$
  such that
  $\limsup_{n \to \infty} \mathbf{E}_{n,\veps}(\mathbf{u}_n) \leq \sigma_\eta
  \mathbf{E}(\mathbf{u})$, and hence
  $\lim_{n \to \infty} \mathbf{E}_{n,\veps}(\mathbf{u}_n) = \sigma_\eta
  \mathbf{E}(\mathbf{u})$ by the liminf inequality.\vskip10pt

  We extend the argument of \cite{garcia2016spectral} for the limsup inequality
  for $E_{n,\veps_n} \overset{\Gamma}{\to} \sigma_\eta E$. Recalling the
  density of Lipschitz functions in $H^1_0(\Omega)$, we assume, without loss of
  generality, that $\mathbf{u}$ is Lipschitz. (Observe that using a
  diagonalization argument for the case that $\mathbf{u}$ is not Lipschitz
  still gives a recovery sequence $\{\mathbf{u}_{n}\}_{n \in \NN}$ with each
  $\mathbf{u}_n \in L^2(V_n; \Sigma_k)$ since $\Sigma_k \subseteq \RR^k$ is
  closed.) Using this assumption, we then produce a recovery sequence
  $\{u_{n,\ell}\}_{n \in \NN}$ for each component $u_\ell$ of $\mathbf{u}$ by
  taking $\{u_{n,\ell}\}_{n \in \NN}$ to be defined by
  $u_{n,\ell} = (u_\ell(x_1), u_\ell(x_2), \ldots, u_\ell(x_n))$ for all
  $n \in \NN$. Since $\mathbf{u} \in H^1_0(\Omega, \rho; \Sigma_k)$, it follows
  that $\mathbf{u}_{n} = (u_{n,1}, u_{n,2}, \ldots, u_{n,k})$ has image in
  $\Sigma_k$ for all $n$. It was shown in \cite{garcia2016spectral} that
  $u_{n,\ell} \overset{\TL^2}{\longrightarrow} u_\ell$ for all $\ell \in [k]$
  and also that $\{u_{n,\ell}\}_{n \in \NN}$ is a recovery sequence for
  $u_\ell$, so it immediately follows that
  $\mathbf{u}_n \overset{\TL^2}{\longrightarrow} \mathbf{u}$ and, again using
  \eqref{obs:sum}, that it satisfies the limsup inequality, so
  $\{\mathbf{u}_{n} = (u_{n,1}, u_{n,2}, \ldots, u_{n,k})\}_{n \in \NN}$ is a
  recovery sequence for $\mathbf{u}$.
    
  \subsection*{Compactness} Claim: Every sequence $\{\mathbf{u}_n\}_{n \in \NN}$ such
  that $\mathbf{u}_n \in L^2_U(V_n; \Sigma_k)$ with
  \[
  \sup_{n \in \NN}\|\mathbf{u}_n\|_{\nu_{n}} < \infty \quad \textrm{and} \quad
  \sup_{n \in \NN} \mathbf{E}_{n, \veps_n}(\mathbf{u}_n) < \infty
  \] 
  is precompact in $\TL^2$.\vskip10pt

Below, we use the  following lemma due to Dejan Slep\v{c}ev. This
    lemma is a consequence of the more general Lemma~\ref{lem:restrict} whose
    proof is given in Appendix~\ref{app:lemma}.

  \begin{lem} \label{lem:VectorDejan}
    In the current setting, if $V \subseteq \Omega$ is a relatively closed set
    with boundary of zero Lebesgue measure and $\mathbf{u}_n
    \overset{TL^2}{\longrightarrow} \mathbf{u}$ as $n \to \infty$, then
    $\chi_V\mathbf{u}_n \overset{TL^2}{\longrightarrow} \chi_V\mathbf{u}$ as $n
    \to \infty$.
  \end{lem}

  Since
  $\|\mathbf{u}_n\|^2 = \|u_{n,1}\|^2 + \|u_{n,2}\|^2 + \cdots +
  \|u_{n,k}\|^2$, the assumption that $\{\mathbf{u}_n\}_{n \in \NN}$ are
  bounded implies that the component functions $\{u_{n,\ell}\}_{n \in \NN}$ are
  bounded for any fixed $\ell \in [k]$. Likewise, the assumption that
  $\{\mathbf{E}_{n, \veps_n}(\mathbf{u}_n)\}_{n \in \NN}$ are bounded implies
  that the energies of the component functions $\{E(u_{n,\ell})\}_{n \in \NN}$
  are bounded for any fixed $\ell \in [k]$. We may then invoke the compactness
  result for $E_{n,\veps_n} \overset{\Gamma}{\to} \sigma_\eta E$ to conclude
  that each sequence $\{u_{n,\ell}\}_{n \in \NN}$ is precompact in
  $\TL^2(\Omega)$ and therefore so is $\{\mathbf{u}_n\}_{n \in \NN}$ in
  $\TL^2(\Omega;\RR^k)$.  Moreover, we have that any limit point
  $\mathbf{u} \in H^1(\Omega, \rho; \RR^k)$, since each component $u_\ell$ has
  finite Dirichlet energy $E(u_\ell)$ by the liminf equality
  \[\sigma_\eta E(u_\ell) \leq \liminf_{n \to \infty} E_{n,\veps_n}(u_{n,\ell})
    < \infty,\] so $u_\ell \in H^1(\Omega, \rho)$, and in fact
  $u_\ell \in H^1_0(U, \rho)$ by the following argument. Supposing that
  $\mathbf{u}_n \overset{\TL^2}{\longrightarrow} \mathbf{u}$ for simplicity,
  using Lemma~\ref{lem:restrict} with $V = \Omega \setminus U$, we have that
  $\mathbf{u}_n \chi_{\Omega \setminus U} \overset{\TL^2}{\longrightarrow}
  \mathbf{u}\chi_{\Omega \setminus U}$ and thus
  $\mathbf{u} = \mathbf{u}\chi_{U}$ since
  $\mathbf{u}_n \chi_{\Omega \setminus U} = 0$ for all $n \in \NN$.

  The only thing left to show is that $\mathbf{u}(\Omega) \subseteq \Sigma_k$
  a.e., but for this we need only to recall that $\mathbf{u}_n
  \overset{\TL^2}{\longrightarrow} \mathbf{u}$ means that $\mathbf{u}_n \circ
  T_n \xrightarrow{L^2(\Omega; \RR^k)} \mathbf{u}$ for some stagnating
  sequence of transportation maps $\{T_n \colon \Omega \to \Omega\}_{n \in
    \NN}$, so the claim is an immediate consequence of convergence in
  $L^2(\Omega;\Sigma_k)$ and $\Sigma_k \subseteq \RR^k$ being
  closed. Therefore $\{\mathbf{u}_n\}_{n \in \NN}$ is precompact in
  $\TL^2$.\qed

  \subsection{Proof of Corollary \ref{cor:Hconv}}  Recalling that any sequence of closed subsets of a
  fixed compact set in $\RR^d$ is precompact with respect to the Hausdorff
  distance \cite[Theorem 1.8.5]{schneider2014convex}, it suffices to prove that
  the only limit points of $\{U_{n,\ell}\}_{n \in \NN}$ are
  $\overline{U_\ell}$.  After passing to a subsequence of
  $\{\mathbf{u}_n\}_{n \in \NN}$, we suppose that 
  $\lim_{n \to \infty} U_{n,\ell} = V_\ell$ for all $\ell \in [k]$. We fix $m \in [k]$ and prove that $V_m = \overline{U_{m}}$.

We first claim that $V_m \supseteq \overline{U_m}$. If not, then there is
$y \in (\Omega \setminus V_m) \cap U_m \cap \{x_n\}_{n \in \NN}$ since
$(\Omega \setminus V_m) \cap U_m$ is of nonzero $\nu$-measure and
$\{x_n\}_{n \in \NN}$ is necessarily dense in $\Omega$ because the conclusion
of Theorem~\ref{thm:GconvDE} holds. Such $y$ is contained in
$U_m \setminus U_{n,m}$, hence $u_{n,m}(y) = 0$, for all sufficiently large
$n > 0$, but $u_m(y) \neq 0$ since $U_m = u_m\inv(0,\infty)$. By the continuity
of $u_m$, we may choose $\veps > 0$ such that
$B_\veps(y) \subseteq \Omega \setminus V_m$ and
\[\|u_{m} - u_{n,m}\circ T_n\|_2^2 \geq 
\int_{B_{\veps}(y)} |u_{m}(x) - u_{n,m}\circ T_n(x)|^2\ dx 
= \int_{B_{\veps}(y)} |u_{m}(x)|^2\ dx >
  0\] 
  for all sufficiently large $n > 0$. But this implies that
\[\lim_{n \to \infty} \|u_{m} - u_{n,m}\circ T_n\|_2^2 \neq 0,\] a
contradiction, and therefore $V_m \supseteq \overline{U_m}$.

To conclude by proving the reverse inclusion, $\overline{U_m} \supseteq V_m $. If $k = 1$, the inclusion follows from $V_\ell \subseteq \overline{U}$ for all $\ell \in [k]$ and $U_1 = U$.  
Let $k > 1$. Since
$\amalg_{\ell \in [k]} U_{n,\ell} = U \cap \{x_n\}_{n \in \NN}$ for all
$n \in \NN$, it follows that $V_{m} \supsetneq \overline{U_{m}}$ would imply
that for some $\ell \neq m$, $V_\ell$ would not contain $U_{n,\ell}$ for
sufficiently large $n > 0$ and thus $V_\ell$ would not contain
$\overline{U_\ell}$, a contradiction. It follows that $V_{m} = \overline{U_{m}}$, and
since $m$ is  arbitrary, we conclude that
$\lim_{n \to \infty} U_{n,\ell}$ both exist and equal $\overline{U_{\ell}}$ for
all $\ell \in [k]$.\qed
\begin{rmk} \label{r:TL2}
  While Corollary \ref{cor:Hconv} states the convergence of the supports of the
  discrete functions $\{u_{n,\ell}\}_{n \in \NN}$ which converge to $u_\ell$ in
  the $TL^2$-sense as $n \to \infty$, it implies that an analogous result holds
  for the supports of their extensions $\{u_{n,\ell}\circ T_n\}_{n \in \NN}$ to
  $\Omega$ which converge to $u_\ell$ in the $L^2(\Omega,\nu)$-sense as $n \to
  \infty$. In particular, if we let $V_{n,\ell} =
  \overline{T_n\inv(U_{n,\ell})}$, then \[\lim_{n \to \infty} V_{n,\ell} =
  \lim_{n \to \infty} U_{n,\ell} = \overline{U_\ell}.\] To see this, it
  suffices to note that by Theorem~\ref{thm:linfinity}, $d_H(U_{n,\ell},V_{n,\ell}) <
  Cr_{(d,n)}$ for sufficiently large $n > 0$, so \[\lim_{n \to \infty}
  d_H(U_{n,\ell},V_{n,\ell}) = \lim_{n \to \infty} Cr_{(d,n)} =0\] and thus the first
  equality holds, with the second equality holding by Corollary \ref{cor:Hconv}.
\end{rmk}

\section{Zaremba Partitions} \label{sec:Zaremba} 
We now consider a modification of Dirichlet partitions, where the Dirchlet
boundary conditions on $\partial U$ have been replaced by Neumann boundary
conditions. We refer to these partitions as Zaremba partitions since eigenvalues of the Laplacian with mixed Dirichlet and Neumann boundary conditions arise. 
We modify our arguments from Section~\ref{sec:MainResultsProof} 
to show that the weighted Zaremba partition problem objective is the
$\Gamma$-limit of the discrete Dirichlet energies when we no longer explicitly
enforce Dirichlet boundary conditions on $V_n \setminus U$. 
However, the regularity results for the continuum Dirichlet partition problem, which hold at least when $\rho \equiv |\Omega|^{-1}$, do not necessarily carry over to the Zaremba partition problem. Qualitative differences between Dirichlet and Zaremba partitions are discussed in Section~\ref{s:Interval} and in Section~\ref{sec:CompEx}.

\subsection{The weighted Zaremba-Laplacian and its spectrum}
For the convenience of the reader, we state some basic results regarding the
spectrum of the weighted Zaremba-Laplacian, by which we mean the operator
$\mathcal{L}\colon u \mapsto -\frac{1}{\rho}\text{div}(\rho^2\nabla u)$
restricted to the Sobolev space
\begin{equation}\label{e:SolbolevZaremba}
H^1(V, \Gamma_D, \rho) := \{u \in H^1(V, \rho):u|_{\Gamma_D}=0\}
\end{equation} 
for an open, bounded, connected Lipschitz domain V and a fixed relatively open
subset $\Gamma_D \subseteq \partial V$; see, {\it e.g.}, \cite[Sections 6.3.1,
8.6]{attouch2014variational}. We say that $(\kappa,u)$ is an eigenpair if $u$
is a weak solution to the system
\begin{align*}
  -\frac{1}{\rho}\text{div}(\rho^2\nabla u) &= \kappa u  \text{ in } V\\
  u &= 0  \ \ \text{ on } \Gamma_D\\
\partial_n u  &= 0  \ \ \text{ on } \Gamma_N := \partial V
                                               \setminus \Gamma_D.
\end{align*}
The standard arguments for the similarly-weighted Dirichlet- and
Neumann-Laplacians (see \cite{garcia2016spectral}) easily extend to the
weighted Zaremba-Laplacian. In particular, the Krein-Rutman theorem
\cite[Theorem 1.2.6]{henrot2006extremum} implies that the first eigenvalue
$\kappa_1$ is simple and that its associated eigenfunction $v_1$ has constant
sign a.e.; as above, we assume eigenfunctions to be positive and have
$L^2(V,\rho)$-norm equal to one.


In the context of Zaremba partitions on $U$, we are only interested in Zaremba
problems of the form $V \subseteq U$ and $\Gamma_D = \partial V \cap
U$. However, we must make sense of the case when $U$ is only quasi-open since
we lack a priori regularity results for $V$. Suppose that $A \subseteq U$ is
quasi-open. We then define
$$ 
H^1_\mathrm{Zar} (U,A, \rho) := \{u \in H^1(U, \rho): u = 0 \text{ q.e. in } U \setminus
A\},$$ which takes the role of
$H^1_0(A, \rho) := \{u \in H^1_0(U, \rho): u = 0 \text{ q.e. in } U \setminus A\}$
in the spectral theory of the Dirichlet-Laplacian on $A$. Observe that if $A$
is open, then $H^1_\mathrm{Zar} (U,A, \rho) = H^1(A,\partial A \cap U, \rho)$, as defined in \eqref{e:SolbolevZaremba}. We
then define the first eigenvalue of the Zaremba-Laplacian on $A$ (with the
prescribed boundary conditions) to be
$$
\kappa_1(A) = \min_{\substack{u \in H^1_\mathrm{Zar} (U,A, \rho) \\
 \|u\|_{L^2(U,\rho)}=1}} \int_{U} |\nabla u(x)|^2\ \rho^2(x) \, dx.
 $$

While the basic spectral theory for the Laplacian is quite similar for the
Neumann-, Dirichlet-, and Zaremba-Laplacians, an important difference is that the 
the Dirichlet-Laplacian eigenvalues have  the monotonicity property:
 if $A \subseteq B$ in the sense of harmonic capacity, then
$\lambda_1(A) \geq \lambda_1(B)$. However, this property fails for the Neumann-Laplacian
\cite[Section 1.3.2]{henrot2006extremum} and thus also fails for the
Zaremba-Laplacian since the Neumann-Laplacian is the  case
$\Gamma_D = \varnothing$. In fact, the counterexample to monotonicity for Neumann eigenvalues involving one rectangle containing another given in \cite[Figure 1.1]{henrot2006extremum} can be modified for Zaremba eigenvalues. Here, one can impose Dirichlet boundary conditions on one of the short sides of both the small and large rectangles. 

For $A \subseteq B$, we have
$H^1_\mathrm{Zar} (U,A, \rho) \subseteq H^1_\mathrm{Zar} (U,B,
\rho)$, which immediately gives the following restricted monotonicity result.
  \begin{prop} \label{prop:Mono}
    If $A, B \subseteq U$ are quasi-open subsets such that $A \subseteq B$ in
    the sense of harmonic capacity, then $\kappa_1(A) \geq \kappa_1(B)$. \qed
  \end{prop}
As the above example shows, monotonicity may fail when we lack an embedding of the relevant Sobolev
space on $A$ into that of $B$ in an $L^2$-norm-preserving manner. 
  
\subsection{Zaremba partitions} \label{s:ZarePartS}
 We define a \textit{Zaremba $k$-partition} to be a collection of $k$ disjoint quasi-open subsets 
 $U_1,U_2,\ldots,U_k$ of $U$ that attains the minimum of
\begin{equation}\label{eq:SumZaremba} 
\sum_{\ell=1}^k \kappa_1(U_\ell).
\end{equation} 
As for Dirichlet partitions, we have the equivalent mapping problem formulation:
\begin{equation}\label{eq:MapProbZaremba}
  \min
  \left\{\mathbf{E}^{\mathrm{Zar}}(\mathbf{u}) \colon \mathbf{u} = (u_1, u_2, \ldots, u_k) \in
    H^1(U, \rho; \Sigma_k), \int_U u_\ell^2(x) \ \rho(x)\, dx = 1 \text{ for all } \ell \in [k]\right\}, 
 \end{equation}
where
\begin{equation*}
\mathbf{E}^{\mathrm{Zar}}(\mathbf{u}) :=
\sum_{\ell=1}^k\int_{U}|\nabla u_\ell(x)|^2\ \rho^2(x)\, dx
\end{equation*}
if $\mathbf{u} \in H^1(U,\rho;\Sigma_k)$ and
$\mathbf{E}^{\mathrm{Zar}}( \mathbf{u}) = \infty$ for all other
$\mathbf{u} \in L^2(U,\rho; \Sigma_k)$.  
Just as the direct methods apply to 
the Dirichlet energy  \eqref{eq:ContDirEnergy} restricted to $H^1_0(U, \rho;\Sigma_k)$,
they may be used over $H^1(U, \rho;\Sigma_k)$ and so we
have that \eqref{eq:MapProbZaremba} has a   solution
$\mathbf{u} = (u_1,u_2,\ldots,u_k)$. As before, we derive the partition
$U_1, U_2, \ldots, U_k$ from $\mathbf{u}$ by taking
$U_\ell = u_\ell\inv(0,\infty)$ for all $\ell \in [k]$. Likewise, we may assume
that $\mathbf{u}$ is nonnegative and quasi-continuous.

Proposition \ref{prop:Mono} implies  the following monotonicity result for Zaremba partitions.  
\begin{lem}\label{prop:ZarembaPartMono}
  The shape functional \eqref{eq:SumZaremba} is monotonic in the sense of
  harmonic capacity, \ie, if $k$-partitions $U_1, U_2, \ldots, U_k$ and
  $V_1, V_2, \ldots, V_k$ are such that $U_\ell \subseteq V_\ell$ in the sense
  of harmonic capacity for all $\ell \in [k]$, then
  \begin{equation*} \sum_{\ell=1}^k \kappa_1(V_\ell) \leq \sum_{\ell=1}^k
\kappa_1(U_\ell).
    \end{equation*}\qed
\end{lem}
Proposition \ref{prop:ZarembaPartMono} implies that a Zaremba $k$-partition satisfies $\overline{U} = \cup_{i=1}^k \overline{U_i}$. As before, this justifies the term ``partition'' in the name. 

Despite the Zaremba partition problem being formally similar to the Dirichlet
partition problem and possessing its montonicity property, it's not clear that
any of the regularity results for the latter carry over to the former. In
particular, though it seems plausible,  it's not clear to us that minimizers
$\mathbf{u}$ of $\mathbf{E}^{\mathrm{Zar}}$ have continuous representatives or
that Zaremba partitions consisting of open sets exist. 
The argument of \cite{caffarelli2007optimal} realizes minimizers of \eqref{eq:ContDirEnergy} as uniform limits of singularly perturbed elliptic equations sharing uniform H\"older limits for all exponents $\alpha \in (0,1)$. 
The obvious modification for Zaremba partitions would be to change the boundary conditions from Dirichlet to Neumann, but it's not clear to us that this argument generalizes. 

\subsection{Partitions of an interval} \label{s:Interval} Another notable
difference the two types of partitions is seen in the case of an interval,
where both problems may be solved exactly. By the monotonicity property and the
connectedness of the sets forming an optimal partition (of both types), we need
only to consider the partitions of the form
$$
U_1 = (0,t_1), U_2 = (t_1, t_1 + t_2), \ldots, U_k = (t_1 + t_2 + \cdots + t_{k-1}, t_1 + t_2 + \cdots + t_{k})
$$ 
where $0 \leq t_\ell \leq 1$ for all $\ell \in [k]$ and $t_1 + t_2 + \cdots + t_k = 1$. 
For an open interval of length $t$,
we have that $\lambda_1 = \pi^2/t^2$ and $\kappa_1 = \pi^2/4t^2$. Thus, the Dirichlet and Zaremba partition problems on $(0,1)$ reduce to
minimizing the functions
$$
f(t_1,t_2,\ldots,t_k) = \pi^2\left(\frac{1}{t_1^2} + \frac{1}{t_2^2} + \cdots
  + \frac{1}{t_k^2}\right)
  $$ 
  and
$$
g(t_1,t_2,\ldots,t_k) = \pi^2\left(\frac{1}{4t_1^2} + \frac{1}{t_2^2} +
  \frac{1}{t_3^2} + \cdots + \frac{1}{t_{k-1}^2} + \frac{1}{4t_k^2}\right),
  $$
respectively, subject to the aforementioned constraints. 
Routine applications of the method of Lagrange multipliers reveal that while the unique Dirichlet
partition is the equipartition 
$$
t_1 = t_2 = \cdots = t_k = 1/k,
$$ 
the unique Zaremba partition has 
\begin{align*}
t_1 &= t_k = 1/(2 + (k-2)\sqrt[3]{4})\\  
t_2 &= t_3 = \cdots = t_{k-1} = \sqrt[3]{4}/(2 +  (k-2)\sqrt[3]{4}). 
\end{align*} 
It follows that $t_2/t_1 = \sqrt[3]{4}$ (independent of $k$), so that the boundary partition components are shorter than interior components. 
We'll further discuss qualitative differences between Dirichlet and Zaremba partitions in Sections~\ref{s:PartDiff} and \ref{sec:discuss}.      
 
\subsection{Consistency results for Neumann boundary conditions on $\partial U$}
We now show that  the Zaremba partition problem is the limit (as $n\to \infty$) of  
the discrete Dirichlet partition problem in which $\Omega = U$ and the side
constraint in \eqref{eq:DiscEnergyGeoGraph} that $u(x_i) = 0 \text{ if } x_i
\in \Omega \setminus U$ is vacuous. 

We construct the graphs $\{G_n\}_{n\in\NN}$ as before, but replace
$\{\mathbf{E}_{n, \veps_n}\}_{n \in \NN}$ with
$\{\mathbf{E}^{\mathrm{Zar}}_{n, \veps_n}\}_{n \in \NN}$ where
$$
\mathbf{E}^{\mathrm{Zar}}_{n, \veps_n}(\mathbf{u}) := \frac{1}{n^2\veps_n^2} \sum_{\ell=1}^k \sum_{i,j=1}^n W_{ij}^n(u_\ell(x_i) - u_\ell(x_j))^2
$$ 
for all
$\mathbf{u}$ in 
\[
L^2(V; \Sigma_k) := \{\mathbf{u}\colon V \to \RR^k \mid \mathbf{u}(V) \subseteq \Sigma_k  \}. 
\] 
Since $\Omega = U$, we have that $L^2(V; \Sigma_k) = L^2_U(V_n; \Sigma_k)$, as defined in \eqref{e:L2U}. We have defined $\mathbf{E}^{\mathrm{Zar}}_{n, \veps_n}(\mathbf{u})$ and $L^2(V; \Sigma_k)$ with a new name to emphasize the difference in boundary conditions, in analogy with the continuum Dirichlet energy and function space defined in Section \ref{s:ZarePartS}. As for the Dirichlet partition setup, the discrete and continuum Dirichlet energies, $\mathbf{E}^{\mathrm{Zar}}_{n, \veps_n}$ and $\mathbf{E}^{\mathrm{Zar}}$, are further extended to the rest of $TL^2(U;\Sigma_k)$ by $\infty$. 

With this setup, Theorem~\ref{thm:GconvZaremba} can be proven using $\Gamma$-convergence arguments componentwise as in Theorem~\ref{thm:GconvDE}. Here, Lemma~\ref{lem:restrict} is no longer needed for establishing the limsup inequality, as there are no longer any explicit pointwise constraints on the vertex functions. As before, Theorem~\ref{thm:GconvZaremba} implies the $\TL^2$-convergence (up to a subsequence) of discrete minimizers $\mathbf{u}_n$ to a continuum minimizer of $\mathbf{u}$. If there exists a continuous representative of the Zaremba ground state, $\mathbf{u}$, the proof of Corollary~\ref{cor:Hconv} would imply the Hausdorff convergence of the discrete Dirichlet partitions (without the auxiliary domain $\Omega$) to the Zaremba partitions (consisting of open sets). 

\section{Computational Examples} \label{sec:CompEx} 
\subsection{An example illustrating consistency}
In this section, we briefly describe the numerical methods used to generate Figures~\ref{fig:cartoon} and
\ref{fig:conv}. We also provide an example to illustrate the difference between
(continuum) Dirichlet and Zaremba partitions.

It has been conjectured (though, to our knowledge, no proof exists) that the
minimal 3-partitions of a disk are rotations of the ``Mercedes star''
\cite{Helffer2010b}. If this assumption holds true, then it is natural to
assume that the Dirichlet 3-partition for the domain in
Figure~\ref{fig:cartoon}(top left),
$$
U = \{ (r,\theta) \colon r< R(\theta) := 1 + 0.3 \cos(3 \theta) \}
$$
is also the ``Mercedes star,'' as illustrated in the top right panel of
Figure~\ref{fig:cartoon}.  The principal Laplace-Dirichlet eigenfunction for
each partition component is plotted. These were  computed using a boundary
integral method  implemented in the Matlab package, \verb+mpspack+
\cite{mpspack}.

We uniformly sampled $n=1584$ points from $\Omega = [-1.5,1.5]^2 \supseteq U$ and constructed 
a weighted geometric graph using the similarity kernel with radial profile $\bm{\eta}(x) = \exp(-x)$ and the admissible sequence $\epsilon_n = n^{-0.3}$. The graph is illustrated in the lower left panel of Figure~\ref{fig:cartoon}. For plotting purposes, we only plot edges with a weight above a fixed threshold.  Finally, we use the
rearrangement algorithm described in \cite{osting2013minimal} to partition the
graph. The partition obtained with the smallest energy is shown in Figure~\ref{fig:cartoon}(lower right). 

To illustrate the consistency statements (Theorem~\ref{thm:GconvDE} and Corollary~\ref{cor:Hconv}),  we repeated this computation for $n=800$, 1600,  and 3200 and plotted the results in Figure~\ref{fig:conv}. The points  sampled from $\Omega \setminus U$ are not displayed and the partition components are colored arbitrarily (so that they do not necessarily agree in different panels of the figure). As $n$ increases, the partition appears to converge to the ``Mercedes star'' partition illustrated in Figure~\ref{fig:cartoon}(top right).

\subsection{A computational comparison between Dirichlet and Zaremba partitions} \label{s:PartDiff}

\begin{figure}[t]
\begin{center}
\includegraphics[width=.32\textwidth]{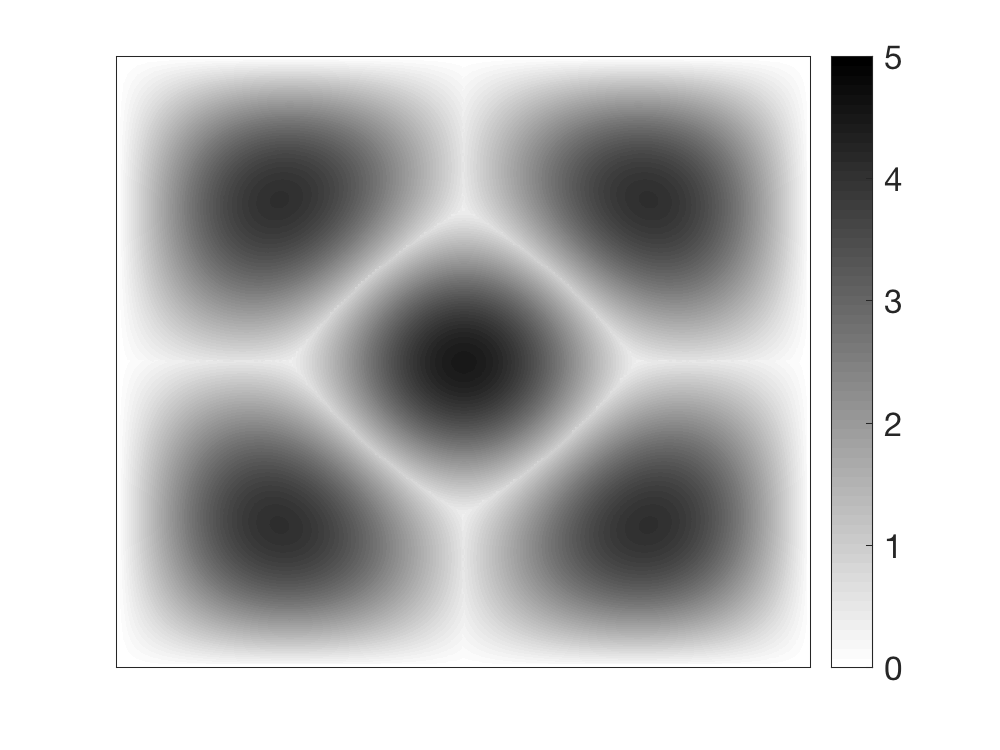}
\includegraphics[width=.32\textwidth]{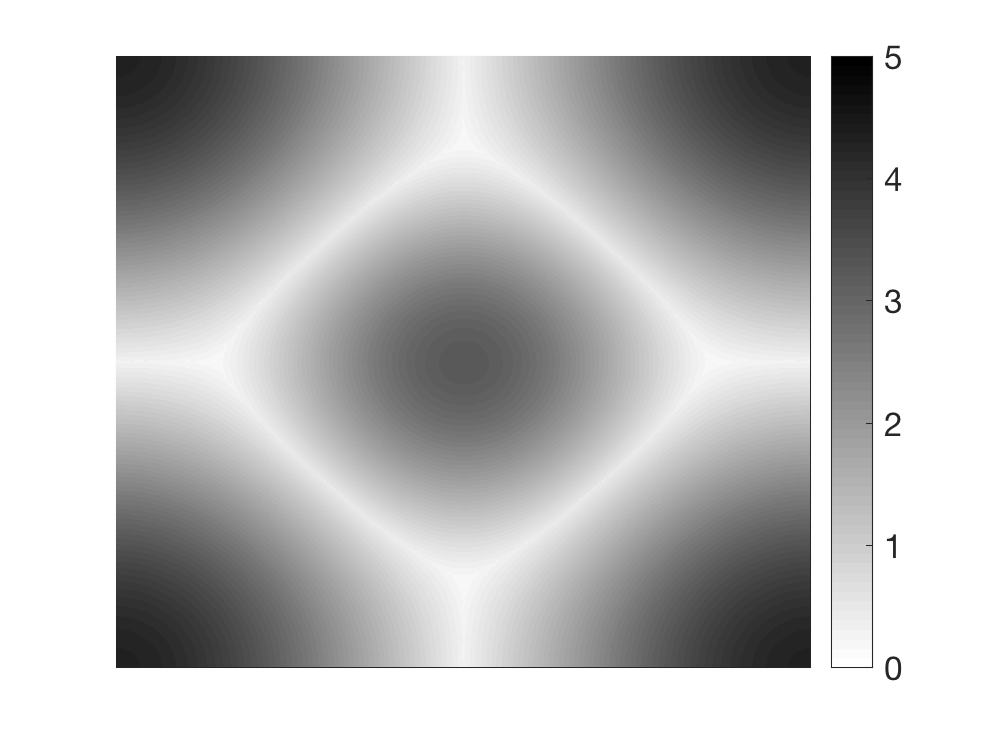}
\includegraphics[width=.32\textwidth]{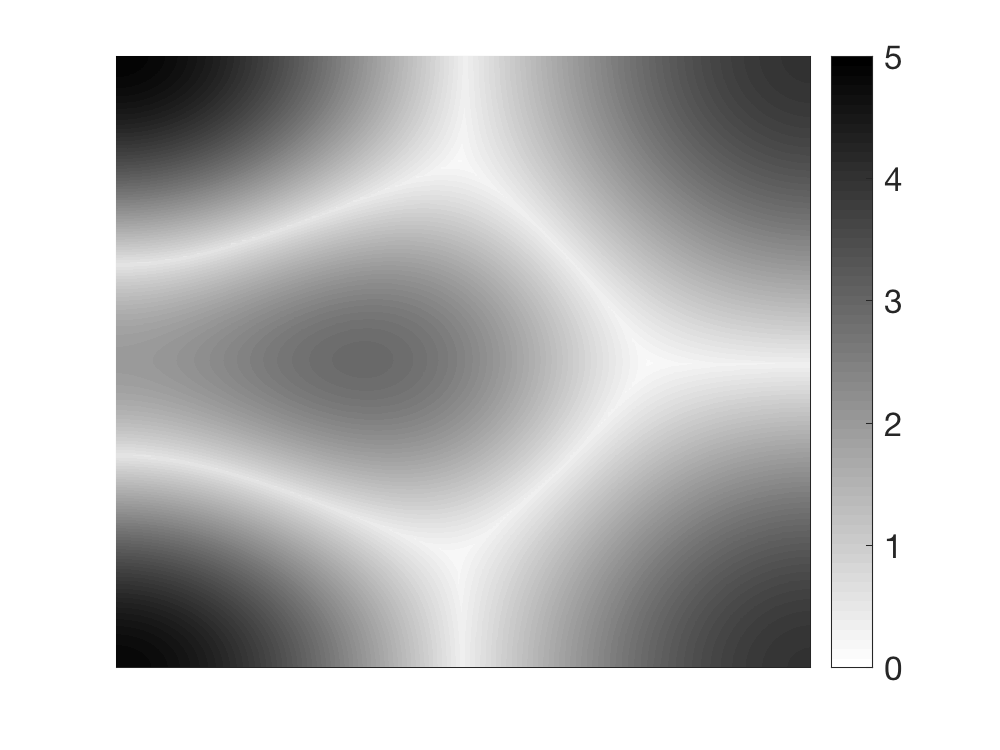} \\
\caption{For a unit square and $k=5$, a comparison of {\bf (left)} a  $k$-Dirichlet partition  and {\bf (center and right)} two locally optimal  $k$-Zaremba partitions. In each component, the first eigenfunction of the Laplacian is plotted with appropriate boundary conditions specified. See Section \ref{s:PartDiff} for a discussion.}
\label{f:ParComp}  
\end{center}
\end{figure}

We consider the problem of approximating Dirichlet partitions
\eqref{eq:ContDirPart} and Zaremba partitions \eqref{eq:SumZaremba} for a unit
square domain, $U = [0,1]^2$. We use the standard 5-point finite difference
approximation of the Laplacian on a $200 \times 200$ square grid with
appropriate boundary conditions. To partition the graph, we use the
rearrangement method described in \cite{osting2013minimal}. The $k$ ground
state components associated with the lowest energy partitions obtained are
plotted in Figure \ref{f:ParComp}. The partition components can be easily
inferred from the supports of the ground state components.  Two local minimum
are found for the Zaremba partitioning problem with similar energies
($\sum_\ell \kappa_1$ is 144.6 for Figure \ref{f:ParComp}(center) and 147.7 for
Figure \ref{f:ParComp}(right)). As for the one-dimensional example in Section
\ref{s:Interval}, we observe that components which intersect the boundary of
$U$ are generally smaller for the Zaremba partition as compared to the
Dirichlet partition. In this and other numerical experiments performed, we
observe that Zaremba partitions generally have more components which intersect
the boundary than Dirichlet partitions. More examples can be found in
\cite{Zosso2015}.

\section{Discussion and Further Directions} \label{sec:discuss} In this paper,
we have proven the consistency statement that the discrete Dirichlet energies
of geometric graphs $\Gamma$-converge to a weighted continuum Dirichlet energy
and, in the case that $\nu$ is the uniform distribution, that the Dirichlet
$k$-partitions of geometric graphs converge to Dirichlet $k$-partitions of the
sampled space in the Hausdorff sense. Our strategy relied on a mapping problem
reformulation due to Caffarelli and Lin \cite{caffarelli2007optimal} for both
the discrete and continuum partitioning problems. We extended results of
Garc\'ia Trillos and Slep\v{c}ev \cite{garcia2016spectral} to show the
$\Gamma$-convergence of the discrete to weighted continuum Dirichlet energies
with respect to the $TL^2$-metric. This, along with a compactness property,
implies the convergence of the ground states. The convergence of the
  ground states, together with the positivity of the ground states on partition
  components, was used to show the Hausdorff convergence of partitions when
  $\nu$ is the uniform distribution. Finally, we also defined a new continuum
partitioning scheme, the Zaremba partition problem, that describes the limiting
behavior (as $n\to\infty$) of the discrete Dirichlet partition problem without
the auxiliary domain $\Omega$ and we proved analogous $\Gamma$-convergence
results.

  In Sections \ref{s:Interval} and \ref{s:PartDiff}, we performed a preliminary
  comparison of Dirichlet and Zaremba partitions, which use two different
  approaches to modeling the boundary of the Euclidean set $U$, from which the
  points are sampled. On one hand, (continuum) Dirichlet partitions seem more
  natural as they equipartition the one-dimensional interval  and appear to
  more closely resemble equipartitions in higher dimensions than Zaremba
  partitions. However, the introduction of an auxiliary domain
  $\Omega \supseteq U$ may not be natural or even possible in all application
  settings. The differences between these partitioning models deserve
  additional attention to specific applications. Perhaps a more natural model
  yet for this consistency result is a closed manifold, where there is no
  boundary.

  An obvious further direction for our theoretical results would be to
  generalize Corollary~\ref{cor:Hconv} to weights $\rho$ other than
  $\rho \equiv |\Omega|^{-1}$ and to Zaremba partitions. Doing so amounts to
  showing that the minimizers $\mathbf{u}$ of $\mathbf{E}$ and
  $\mathbf{E}^{\mathrm{Zar}}$ admit continuous representatives, as was done for
  the original mapping problem in \cite{caffarelli2007optimal}. Perhaps one
  could then also extend the regularity results of optimal partitions to these
  generalized settings. We remark that this discussion also applies to the case
  of the symmetric normalized Dirichlet energies; see Remark~\ref{rmk:normalized}.

Another direction for future theoretical work is to prove similar
results for related partitioning schemes. In this paper, we focused on the
$\ell^1$-norm of the vector of eigenvalues
$(\lambda_1(U_\ell))_{1 \leq \ell \leq k}$, but the $\ell^p$-norm for
$1 \leq p \leq \infty$ is also considered in shape optimization literature,
\eg, $p = \infty$ gives the problem of minimizing $\max_\ell \lambda(U_\ell)$
over partitions $\amalg_{\ell\in[k] U_\ell}$
\cite{helffer2010remarks}. Likewise, since the graph $p$-Laplacian has been
applied to machine learning \cite{luo2010plaplacian}, one could partition using
the eigenvalues of the graph $p$-Laplacian for $1 \leq p < \infty$. Our results
should extend to this clustering scheme if one could prove $\Gamma$-convergence
results analogous to those of \cite{garcia2016spectral}, \eg, that the discrete
$p$-Dirichlet energy
$$u \mapsto \sum_{i,j}^nW_{ij}|u(x_i)-u(x_j)|^p$$ $\Gamma$-converges to the
continuum $p$-Dirichlet energy
$$u \mapsto \int_U |\nabla u|^p\ dx$$ in the $TL^p$-sense. In particular, when
$p > d$, the continuity of continuous minimizers would immediately follow from
the Sobolev embedding theorems and so an analogue of Corollary~\ref{cor:Hconv}
would also hold. Another, less overtly similar clustering scheme is to take a
nonnegative matrix factorization (NMF) of a matrix associated to a graph. In
\cite[Proposition 2.1]{osting2013minimal} it was shown that the Dirichlet
partition problem using the eigenvalues of the random walk factorization
$D^{-1}W$ is equivalent to an NMF of the matrix $D^{-1/2}WD^{-1/2}$, so a
consistency statement for this NMF problem could be obtained from one involving
the random-walk Laplacian.

Finally, in applications that demand graph partitions of very large datasets, it is common to subsample the edges and/or
vertices of a graph, which is sometimes referred to as graph sparsification or the identification of a coreset. The consistency results in the paper supports this practice, but quantifying the error incurred would require a convergence rate of the Dirichlet partitions, a problem we view as difficult.

\appendix
\section{Lemma on restricting sequences}\label{app:lemma}

The following lemma, due to Dejan Slep\v{c}ev,  shows that the convergence of a sequence $(\mu_n, f_n)
\to  (\mu,f)$ in $\TL^p(\RR^d)$ is preserved upon
restricting the functions to a subset $\Omega \subseteq \RR^d$. 
We use the analogue for $\TL^p(\RR^d;\RR^k)$, stated as Lemma~\ref{lem:VectorDejan}, for which the proof follows mutatis mutandis.

\begin{lem}\label{lem:restrict}
  Let $p \geq 1$ and  $\Omega \subseteq \RR^d$ be such that $\mu(\partial
  \Omega)=0$. Assume $\mu_n$ and $\mu$ are probability measures on $\RR^d$
  and $(\mu_n, f_n) \overset{\TL^p}{\longrightarrow} (\mu,f)$ as $n \to
  \infty$. Then $(\mu_n, f_n \chi_\Omega) \overset{TL^p}{\longrightarrow}
  (\mu,f \chi_\Omega)$ as $n \to \infty$. 
\end{lem}
\begin{proof}
  Let $\veps>0$.  There exists $n_1$ such that for all $n \geq n_1$, $$
  d_{\TL^p}((\mu_n,f_n), (\mu,f))^p < \frac{\veps}{4 \cdot 2^p}.$$  Since
  \[\lim_{M \to \infty} \int_{\{x \colon  |f(x)|>M\}} |f(z)|^p\ d\mu(z) = 0,\]
  there exists $\alpha>0$ such that for all measures $\sigma$, with $0\leq
  \sigma \leq \mu$ and $\sigma(\RR^d)\leq \alpha$ it holds that
  \begin{equation} \label{est_sig}
    \int_{\RR^d} |f(z)|^p\ d\sigma(z) < \frac{\veps}{8 \cdot 2^p}.
  \end{equation}
  Given a set $A \subseteq \RR^d$ and $\delta>0$, let $A_\delta$ be the thickened
  set $A_\delta= \{ x \in \RR^d \colon  d(x, A)< \delta\}$. Since
  $\partial \Omega = \bigcap_{\delta>0} (\partial \Omega)_\delta$, it follows
  that $$\lim_{\delta \to 0} \mu((\partial \Omega)_\delta) = \mu(\partial
  \Omega) = 0.$$ Therefore there exists $\delta>0$ such that $\mu((\partial
  \Omega)_\delta) < \alpha$.

For any $n$, let $\pi_n \in \Pi(\mu_n, \mu)$ be a transportation plan
  such that
  \[ \int_{\RR^d \times \RR^d} |f_n(x) - f(y)|^p + |x-y|^p\ d\pi_n(x,y) < 2 \,
  d_{TL^p}((\mu_n,f_n), (\mu,f))^p.  \] Since $(\mu_n, f_n)
  \overset{TL^p}{\longrightarrow} (\mu,f)$, there exists $n_2$ such that for all
  $n \geq n_2$
  \begin{equation} \label{est_pi} \pi_n(\{(x,y) \colon  |x-y|> \delta \}) <
    \alpha.
  \end{equation}
 For all $n \geq \max\{ n_1, n_2\}$
  \begin{align*}
    d_{TL^p}((\mu_n,f_n \chi_\Omega), (\mu,f \chi_\Omega))^p \leq \,&
    \int_{\RR^d \times \RR^d} |f_n(x) - f(y)|^p + |x-y|^p\ d\pi_n(x,y) \\
    &+ \int_{\Omega \times (\RR^d \backslash \Omega)} |f_n(x)|^p\ d\pi_n(x,y)
    + \int_{(\RR^d \backslash \Omega) \times \Omega}  |f(y)|^p\ d\pi_n(x,y) \\
    \leq \,  & 2  d_{TL^p}((\mu_n,f_n), (\mu,f))^p \\
    & + \int_{\{(x,y) \colon  |x-y| \geq \delta\} \cup (\partial \Omega)_\delta}
    |f_n(x)|^p + |f(y)|^p \ d\pi_n(x,y).
  \end{align*}
  Let ${E^\delta} := {\{(x,y) \::\: |x-y| \geq \delta\} \cup (\partial
    \Omega)_\delta} $.  By our choice of $\delta$, using \eqref{est_sig} and
  \eqref{est_pi}, it follows that $$\int_{E^\delta} |f(y)|^p \ d\pi_n(x,y) <
  \frac{\veps}{4 \cdot 2^p}.$$  To estimate the integral of $f_n$, note that
  \begin{align*}
    \int_{E^\delta}  |f_n(x)|^p \ d\pi_n(x,y) 
    & \leq \int_{E^\delta}  (|f_n(x) - f(y) | + |f(y)|)^p\ d\pi_n(x,y)  \\
    & \leq \,    \int_{E^\delta}   2^p (|f_n(x)|- f(y) |^p + |f(y)|^p )\ d\pi_n(x,y)  \\
    & \leq \,  2^{p+1} d_{TL^p}((\mu_n,f_n), (\mu,f))^p + \frac{\veps}{4 \cdot
    }.
  \end{align*}
  Combining the estimates gives 
  \[ 
  d_{TL^p}((\mu_n,f_n \chi_\Omega), (\mu,f \chi_\Omega)) < \, 2^{p+2}
  d_{TL^p}((\mu_n,f_n), (\mu,f))^p + \frac{\veps}{4 \cdot 2^p}+ \frac{\veps}{4
    \cdot } < \veps.  \]
\end{proof}

\printbibliography
\end{document}